\documentclass[preprint,hidelinks,11pt]{elsarticle}
\usepackage{psfrag,graphicx,amsmath,amsthm,amssymb,hyperref,color}
\journal{}
\usepackage{multirow}
\usepackage{subfigure}
\newtheorem{example}{Example}[section]

\newtheorem{theorem}{Theorem}[section]

\newtheorem{lemma}{Lemma}[section]
\newtheorem{corollary}{Corollary}[section]
\numberwithin{equation}{section}
\DeclareFontEncoding{FMS}{}{}
\DeclareFontSubstitution{FMS}{futm}{m}{n}
\DeclareFontEncoding{FMX}{}{}
\DeclareFontSubstitution{FMX}{futm}{m}{n}
\DeclareSymbolFont{fouriersymbols}{FMS}{futm}{m}{n}
\DeclareSymbolFont{fourierlargesymbols}{FMX}{futm}{m}{n}
\DeclareMathDelimiter{\tbar}{\mathord}{fouriersymbols}{152}{fourierlargesymbols}{147}
\begin{document}
\begin{frontmatter}
%%%%%%%%%%%%%%%%%%%%%%%%%%%%%%%%%%%%%%%%%%%%%%%%
\title{\textbf{Local discontinuous Galerkin methods for the time tempered fractional diffusion equation}}

%%%%%%%%%%%%%%%%%%%%%%%%%%%%%%%%%%%%%%%
\author{Xiaorui Sun\fnref{label1}}

\author{Can Li\fnref{label1}}
\ead{mathlican@xaut.edu.cn}

\author{Fengqun Zhao \fnref{label1}}
\ead{fqzhao@xaut.edu.cn}

\address[label1]{Department of Applied Mathematics, School of Sciences,
 Xi'an University of Technology, Xi'an, Shaanxi 710054, P.R.China}

\begin{abstract}
In this article, we consider discrete schemes for a fractional diffusion equation involving a tempered fractional derivative in time. We present a semi-discrete scheme by using  the local discontinuous Galerkin (LDG) discretization in the spatial variables. We prove that the semi-discrete scheme is unconditionally stable in $L^2$ norm and convergence with optimal convergence rate $\mathcal{O}(h^{k+1})$. We develop a class of fully discrete LDG schemes by combining the weighted and shifted Lubich difference operators with respect to the time variable, and establish the  error estimates.
Finally, numerical experiments are presented to verify the theoretical results.
\end{abstract}


\end{frontmatter}

\section{Introduction}
\label{sec:0}

 In this paper we discuss a local discontinuous Galerkin method to
solve the following  time tempered fractional subdiffusion equation
\cite{Henry:06,Langlands:08}
\begin{equation}\label{orgproblema}
u_t(x,t)=\kappa_\alpha ~_{0}D_t^{1-\alpha,\lambda}\left(u_{xx}(x,t)\right)-\lambda u(x,t),
\end{equation}
where $u(x,t)$ represents the probability density of finding a
particle on $x$ at time $t$,  $\kappa_{\alpha}>0$ is the diffusion coefficient, and $_{0}D_t^{1-\alpha,\lambda}(0<\alpha<1)$ denotes the Riemann-Liouville tempered fractional derivative operator. The  Riemann-Liouville tempered fractional derivative  of order $\gamma ~(n-1<\gamma<n)$ is defined by \cite{Podlubny:99,Sabzikar:15,Li:14}
\begin{equation}\label{suTRLD}
_{0}D_t^{\gamma,\lambda}u(t)
= { _{0}D}_t^{n,\lambda}{_{0}I_t^{n-\gamma,\lambda}}u(t),
\end{equation}
where ${_{0}I_t^{n-\alpha,\lambda}}$ denotes the Riemann-Liouville fractional tempered integral \cite{Podlubny:99,Sabzikar:15,Li:14}
\begin{equation}\label{TRLI}
_{0}I_t^{\sigma,\lambda}u(t)=\frac{1}{\Gamma(\sigma)}\int_{0}^te^{-\lambda( t-s)}(t-s)^{\sigma-1}u(s)ds,\sigma=n-\alpha,
\end{equation}
and
\begin{equation}\label{ddlammda}
_{0}D_t^{n,\lambda}=\bigg(\frac{d}{d t}+\lambda \bigg)^n=\underbrace{\bigg(\frac{d}{d t}+\lambda \bigg)\cdots\bigg(\frac{d}{d t}+\lambda \bigg)}_{n\ times}.
\end{equation}
Tempered fractional calculus can be recognized as the generalization of fractional calculus. If we taking $\lambda=0$ in (\ref{suTRLD}), then the tempered fractional integral and derivative operators reduce  to the Riemann-Liouville
fractional integral ${_{0}I_t^{\sigma}}$ and derivative $_{0}D_t^{\gamma}$ operators, respectively.
\\
In recent years, many numerical methods such as finite difference methods \cite{Meerschaert:04,Sun:16,Murillo:15,Liu:15,Sousa:15,Gracia:15}, finite element methods \cite{Ervin:05,Wang:14,Li:15a,Zhao:15,Jin:16} and spectral methods \cite{Lin:07,Wang:16} have been developed for the numerical solutions of fractional subdiffusion and superdiffusion equations. Limited works are reported for solving the  tempered fractional differential equations, when compared with a large volume of literature on numerical solutions of fractional differential equations. In this literature, Baeumera and Meerschaert \cite{Baeumera:10} provide finite difference and particle tracking methods for solving the tempered fractional diffusion equation with the second order accuracy. The stability and convergence of the provided schemes are discussed.
Cartea and del-Castillo-Negrete \cite{Cartea:07a} derive a finite difference scheme to numerically solve a Black-Merton-Scholes model with tempered fractional derivatives.
 Hanert and Piret
\cite{Hanert:02} presented a Chebyshev pseudo-spectral scheme to solve the space-time tempered fractional diffusion equation, and proved that the method yields an exponential convergence rate. Zayernouri et al. \cite{Zayernouri:15}  derived an efficient Petrov-Galerkin method for solving tempered fractional ODEs by using the eigenfunctions
of tempered fractional Sturm-Liouville problems.  By using the weighted and shifted Gr$\ddot{\rm u}$nwald difference (WSGD) operators,
Li and Deng \cite{Li:14} designed a series of high order numerical schemes for the space tempered fractional diffusion equations. This technique is used to solve the tempered fractional Black-Scholes equation for European double barrier option \cite{Zhang:17}. Using the properties of generalized Laguerre functions, Huang et al. \cite{Huang:14} used Laguerre functions to approximate the substantial fractional ODEs on the half line. Li et al. \cite{Li:15}  analysed the well-posedness and developed the Jacobi-predictor-corrector algorithm for the tempered fractional ordinary differential equation.
Yu et al. \cite{Yu:06} developed the third and fourth order quasi-compact approximations for one and two dimensional space tempered fractional diffusion equations.
By using the weighted and shifted Gr\"{u}nwald-Letnikov formula suggested in \cite{Tian:15}, Hao et al. \cite{Hao:17} constructed a second-order approximation for the time tempered fractional diffusion equation.
By introducing fractional integral spaces, Zhao et al. \cite{Zhao:16} discussed spectral
Galerkin and Petrov-Galerkin methods for tempered fractional advection
problems and tempered fractional diffusion problems. \\
Recently, the high order and fast numerical methods for fractional differential equations  draw the wide interests of the researchers \cite{Tian:15,Chen:14,Chen:15,Wang:10,Jiang:17}. The local discontinuous Galerkin (LDG) method is one of the most popular methods in this literature. The  LDG method was first introduced to solve a convection-diffusion problems by Cockburn and Shu \cite{Cockburn:98}. These methods  have recently become increasingly popular
due to their flexibility for adaptive simulations, suitability for parallel computations, applicability to problems with discontinuous  solutions, and compatibility
with other numerical methods.
Nowadays, the LDG method has been successfully used in solving linear and
nonlinear elliptic, parabolic, hyperbolic equations, and
some mixed schemes. For the recent development of discontinuous Galerkin methods, see the the monograph and review articles \cite{Hesthaven:08,Cockburn:00,Xu:10,Shu:16}
 and the references therein. More recently, some researchers pay attention to solving the fractional partial equations by the LDG method.
For the time fractional  differential equations,  Mustapha and McLean \cite{Mustapha:10} employed a piecewise-linear, discontinuous Galerkin method for the
time discretization of a sub-diffusion equation.  A  LDG method for space discretization of a time fractional diffusion equation is discussed in  Xu and Zheng's work \cite{Xu:13}. By using L1 time discretization, Wei et al.\cite{Wei:12} developed an implicit fully discrete LDG finite element method for solving the time-fractional Schr\"{o}inger equation, Guo et al. \cite{Guo:16} studied a LDG method for some time fractional fourth-order differential equations. Liu et al. \cite{Liu:14} proposed LDG method combined with a third order weighted and shifted Gr\"{u}nwald difference operators  for a fractional subdiffusion equation.
For the space fractional  differential equations, based on two (or four) auxiliary
variables in one dimension (or two dimensions) and the Caputo derivative as the spatial
derivative, Ji and Tang \cite{Ji:12}  developed the high-order accurate Runge-Kutta LDG  methods for one- and two-dimensional space-fractional diffusion equations
with the variable diffusive coefficients.
Deng and Hesthaven \cite{Deng:13} have developed a LDG method
for  space fractional diffusion equation and given a fundamental frame to combine the LDG methods with fractional operators.
 \\
In this paper,  we will develop and analyze a new class of LDG method for the model (\ref{orgproblema}). Our new  method is based on a combination of the weighted and shifted Lubich difference approaches in
the time direction and a LDG method in the space direction. Stability and convergence of semi-discrete and fully discrete LDG schemes are rigorously analyzed. We  show  that the fully discrete scheme is unconditionally stable with convergence order of $\mathcal{O}(\tau^q+h^{k+1}),q=1,2,3,4,5$.\\
The rest of the article is organized as follows. In section \ref{sec:2}, we first consider the initial boundary value problem of the tempered fractional diffusion equation. Then in section \ref{sec:3} we construct a semi-discrete LDG method for the considered equation.
We perform the detailed theoretical analysis for the stability and error estimate
of the semi-discrete numerical scheme in this section.
In section \ref{sec:4}, we  apply the weighted and shifted Lubich difference approximation for the temporal discretization of the time fractional tempered equation.
The error estimates are provided for the full-discrete LDG scheme.
Finally, some numerical examples and physical simulations are presented in section \ref{sec:5} which confirm the theoretical statement.
 Some concluding remarks are given in the final section.
%%%%%%%%%%%%%%%%%%%%%%%%%%%%%%%%%%%%%%%%%.
%%%%%%%%%%%%%%%%%%%%%%%%%%%%%%%%%%%%%%%%%
\section{Initial boundary value problem of the tempered fractional diffusion equation}\label{sec:2}
Instead of  designing the numerical scheme of the equation (\ref{orgproblema}) directly, we constructing the numerical scheme for its equivalent form.
By simple calculation, the equation (\ref{orgproblema}) can be rewritten as
\begin{equation}\label{orgproblemaa}
\frac{d}{dt}\left( e^{\lambda t}u(x,t)\right)=\kappa_{\alpha} ~_{0}D_t^{1-\alpha}\left(e^{\lambda t}u_{xx}(x,t)\right).
\end{equation}
Performing Riemann-Liouville fractional integral $_{0}I_t^{1-\alpha}$ on both side of \eqref{orgproblema}, we arrive at
\begin{equation}\label{orgproblemba}
{_{0}^{C}D}_t^{\alpha}\big(e^{\lambda t}u(x,t)\big)=\kappa_{\alpha} e^{\lambda t}u_{xx}(x,t),
\end{equation}
where ${_{0}^{C}D}_t^{\alpha}$ denotes the Caputo  fractional derivative \cite{Podlubny:99}
    \begin{equation}\label{TCDcc}
      {_{0}^{C}D}_t^{\alpha}u(t)=\frac{1}{\Gamma(1-\alpha)}
      \int_{0}^t\frac{1}{(t-s)^{\alpha}}\frac{d u(s)}{d
          s}ds,0<\alpha<1.
    \end{equation}
In view of the Caputo tempered fractional derivative \cite{Sabzikar:15,Li:14}
    \begin{equation}\label{TCDcc}
      {_{0}^{C}D}_t^{\alpha,\lambda}u(t)=e^{-\lambda t}~{_{0}^{C}D}_t^{\alpha}\big(e^{\lambda t}u(t)\big)=\frac{e^{-\lambda t}}{\Gamma(1-\alpha)}\int_{0}^t\frac{1}{(t-s)^{\alpha}}\frac{d (e^{\lambda s}u(s))}{d
          s}ds,
    \end{equation}
we get the following time tempered fractional diffusion equation
\begin{equation}\label{1.1}
{_{0}^{C}\!D}_{t}^{\alpha,\lambda}u(x,t)=\kappa_{\alpha}u_{xx}(x,t).
\end{equation}
Let $\Omega= (0,L)$ be the space domain. We now
consider the initial boundary value problem of the fractional tempered diffusion equation (\ref{1.1})
in the domain $(0,L)\times(0,T]$, subject to the initial condition
\begin{equation}\label{intinalvalue}
u(x,0)=u_0(x),~x\in \Omega£¬
\end{equation}
and the boundary conditions
\begin{equation}\label{Bvalue}
u(x,t)=u(x+L,t),L>0,t\in (0,T].
\end{equation}
%%%%%%%%%%%%%%%%%%%%%%%%%%%
\begin{lemma}  \cite{Alikhanov:01} \label{lemma1}
For function $u(t)$absolutely continuous on $[0,t]$, holds the the following
 inequality
\begin{equation}\label{3.1inq}
 \frac{1}{2}{_{0}^{C}\!D}_{t}^{\alpha}(u^2(t))\leq
 u(t){_{0}^{C}\!D}_{t}^{\alpha}(u(t)),0<\alpha<1.
\end{equation}
\end{lemma}
For the initial-boundary value problem (\ref{1.1})-(\ref{Bvalue}), we have the following stability.
%%%%%%%%%%%%%%%%%%%%%%%%%%%%%
\begin{theorem}\label{stabilityLABCL2}
The solution $u(x, t)$ of the initial-boundary value  problem (\ref{1.1})-(\ref{Bvalue}) holds the prior estimate
\begin{align} \label{IQ1thm}
 \|u(\cdot,t)\|^2+2\kappa_\alpha~{_{0}I}_t^{\alpha,2\lambda}\|u_x(\cdot,t)\|^{2}
 \leq e^{-2\lambda t}\|u_0\|^{2},
\end{align}
where ${_{0}I}_t^{\alpha,2\lambda}$ denotes  the Riemann-Liouville fractional tempered integral operator defined by (\ref{TRLI}).
\end{theorem}
\begin{proof}
Taking $v(x,t)=e^{\lambda t}u(x,t)$ in (\ref{1.1}), we have
\begin{align}
{_{0}^{C}D}_t^{\gamma}v(x,t)=\kappa_\alpha v_{xx}(x,t),   &x\in \Omega, \;t>0.   \label{exacta1eqiq}
\end{align}
Taking  inner product in equation (\ref{exacta1eqiq}) with the space variable, we get
\begin{equation}\label{integralAiq}
   \big( v(\cdot,t),{_{0}^{C}D}_t^{\gamma}v(\cdot,t)\big)_\Omega=-\kappa_\alpha\big( v_x(x,t), v_x(x,t)\big)_\Omega +\kappa_\alpha v_x(x,t) v(x,t)\big|^{L}_{0},
\end{equation}
where $(u(x),v(x))_\Omega=\int_\Omega u(x)v(x)dx$. Denoting the $L^2$-norm as $\|u\|=(u(x),u(x))_\Omega^{1/2}$,
using the boundary conditions (\ref{Bvalue}) and using the inequality  (\ref{3.1inq}), we get
\begin{align}\label{fiintegralCiq}
{_{0}^{C}D}_t^{\alpha}\| v(\cdot,t)\|^2+2\kappa_\alpha\| v_x(\cdot,t)\|^2\leq 0.
\end{align}
By applying the fractional integral  operator ${_{0}I}_t^{\alpha}$ to both sides of inequality (\ref{fiintegralCiq}),
using the composite properties of fractional calculus \cite{Podlubny:99}
$${_{0}I}_t^{\alpha}~{_{0}D}_t^{\alpha}(w(t))=w(t)-w(0),$$
 we
obtain
\begin{align}\label{finalest}
\|v(\cdot,t)\|^2+2\kappa_\alpha~{_{0} I}_t^{\alpha}\|v_x(\cdot,t)\|^{2}\leq \|v_0\|^2.
\end{align}
Taking $u(x,t)=e^{-\lambda t}v(x,t)$ in (\ref{finalest}) we get  the desired estimate \eqref{IQ1thm}.
\end{proof}
%%%%%%%%%%%%%%%%%%%%%%%%%%%%%%%%%%%%%%%%%%
%%%%%%%%%%%%%%%%%%%%%%%%%%%%%%%%%%%%%%%%%%
\section{The semi-discrete LDG scheme}
\label{sec:3}
In this section, we present and analyze a local discontinuous
Galerkin method for the equation (\ref{1.1}) subjects to
the initial condition (\ref{intinalvalue}) and the periodic boundary
conditions (\ref{Bvalue}).
For the interval $\Omega=[0,L]$, we divide it into $N$ cells as follows $$0=x_{1/2}<x_{3/2}<\cdots<x_{N-1/2}<x_{N+1/2}=L.$$
We denote
$I_j=(x_{j-1/2},x_{j+1/2}),x_{j}=(x_{j-1/2}+x_{j+1/2})/2,$
and
$h_j=x_{j+1/2}-x_{j-1/2},h={\rm max}_{1\leq j\leq N}h_j.$
Furthermore, we define the mesh $\mathcal{T}=\{I_j=(x_{j-1/2},x_{j+1/2}),j=1,2,...,N\}$.
The finite element space is defined by
$$V_h=\big\{v: v\mid_{I_j}\in
P^k(I_j),j=1,2,\cdots,N\big\},$$ where $P^k(I_j)$ denotes the set of all
 polynomials of  degree at most $k$ on cell $I_j$. We define
$u^{-}_{j+1/2},u^{+}_{j+1/2}$ represent the values
of $u$ at $x_{j+1/2}$ from the left cell $I_{j}$ and the right cell $I_{j+1}$,
respectively.\\
To define the local discontinuous Galerkin method, we rewrite (\ref{1.1}) as a first-order system
\begin{equation}\label{2.3}
\begin{array}{l}
\displaystyle
 {_{0}^{C}\!D}_{t}^{\alpha,\lambda}u(x,t)-\kappa_{\alpha}p_x=0,\\
\displaystyle
p-u_x=0.
\end{array}
\end{equation}
Now we can define the local
discontinuous Galerkin method to the system (\ref{2.3}).
Find $u_h,p_h$,  for all $v_h, w_h\in V_h$ such that
\begin{equation}\label{semi-discreteLDG}
\begin{array}{l}
\displaystyle
\big({_{0}^{C}\!D}_{t}^{\alpha,\lambda}u_h, v_h\big)_{I_j}+\kappa_{\alpha}\left(\big(p_h,(v_h)_x\big)_{I_j} -(\widehat{p}_hv_h^{-})_{j+\frac{1}{2}}
+(\widehat{p}_hv_h^{+})_{j-\frac{1}{2}}\right)=0, \\
\displaystyle
\big(p_h ,w_h\big)_{I_j}+\big(u_h ,(w_h)_x\big)_{I_j} - (\widehat{u}_h w_h^{-})_{j+\frac{1}{2}}
+(\widehat{u}_h w_h^{+})_{j-\frac{1}{2}}=0.
\end{array}
\end{equation}
The 'hat' terms in (\ref{semi-discreteLDG}) are the numerical probability density
fluxes. We chose the alternating numerical fluxes \cite{Cockburn:00,Kirby:14}  as
\begin{equation}\label{flux2.5}
\widehat{u}_h=u_h^{-},\qquad\widehat{p}_h=p_h^{+}.
\end{equation}
%%%%%%%%%%%%%%%%%%%%%%%%%%%%%%%%%%%
\subsection{Stability  analysis of the semi-discrete LDG scheme}\label{sec:3.1}
In this section, we present the stability and convergence analysis for
the semi-discrete scheme (\ref{semi-discreteLDG}) in $L^2$ sense. To do so, we follow
the technique used by Cockburn and Shu \cite{Cockburn:98}.
%%%%%%%%%%%%%%%%%%%%%%%%%%%%%%%%%%%%%%
\begin{lemma} \label{lemma1}
For function $u(t) $absolutely continuous on $[0,t]$, holds the the following
 inequality
\begin{equation}\label{3.1inportantiq}
e^{-\lambda t}\frac{1}{2}  \; _{0}^{C}\!D_{t}^{\alpha,\lambda} \big(u(t)\big)^2 \leq u(t)~ _{0}^{C}\!D_{t}^{\alpha,\lambda}\big(u(t)\big),  0< \alpha <1,\lambda>0.
\end{equation}
\end{lemma}
\begin{proof}
Taking $v(t)=e^{\lambda t}u(t)$ in the inequality (\ref{3.1inq}), we have
\begin{equation}
\frac{1}{2}  \; _{0}^{C}\!D_{t}^{\alpha} \big(e^{\lambda t}u(t)\big)^2 \leq e^{\lambda t}u(t)~ _{0}^{C}\!D_{t}^{\alpha}\big(e^{\lambda t}u(t)\big) .
\end{equation}
which leads
\begin{equation}
 e^{-2\lambda t}\frac{1}{2}  \; _{0}^{C}\!D_{t}^{\alpha} \big(e^{\lambda t}u(t)\big)^2
 \leq u(t)~ _{0}^{C}\!D_{t}^{\alpha,\lambda}\big(u(t)\big).
\end{equation}
\end{proof}
%%%%%%%%%%%%%%%%%%%%%%%%%%%%%%%%%%%%%%%%%%%
\begin{lemma}\label{lemma2}
 Let $u_h,p_h$ to be the
solution of semi-discrete LDG scheme (\ref{semi-discreteLDG}) with the flux $\widehat{u}_h$ ,$\widehat{p}_h$
defined in (\ref{flux2.5}), holds
\begin{equation}\label{3.2}
 \big(u_h,\,{_{0}^{C}\!D}_{t}^{\alpha,\lambda}u_h\big)_{\Omega}
 +\kappa_{\alpha}\|p_h\|^2=0.
\end{equation}
\end{lemma}
%%%%%%%%%%%%%%%%%%%%%%%%%%%%%%%%%%%%%%
\begin{proof} For simplicity, we denote
\begin{eqnarray}\label{3.3}
B(u_h,p_h;v_h,w_h)&:=&\big({_{0}^{C}\!D}_{t}^{\alpha,\lambda}(u_h),v_h\big)_{\Omega}
 + \kappa_{\alpha}\big(\big(p_h,(v_h)_x\big)_{\Omega}   \nonumber\\
&-&\sum_{j=1}^{N}\big[(\widehat{p}_h)_{j+\frac{1}{2}}v^{-}_{j+\frac{1}{2}}-(
\widehat{p}_h)_{j-\frac{1}{2}}(v_h)^{+}_{j-\frac{1}{2}}\big]\big) \\
&+& \big(u_h ,(w_h)_x\big)_{\Omega}+\big(p_h, w_h\big)_{\Omega}   \nonumber \\
&-&\sum_{j=1}^{N}\big[(\widehat{u}_h)_{j+\frac{1}{2}} (w_h)^{-}_{j+\frac{1}{2}}
-(\widehat{p}_h)_{j-\frac{1}{2}}(w_h)^{+}_{j-\frac{1}{2}}\big]
\nonumber.
\end{eqnarray}
If we take $v_h=u_h, w_h=\kappa_{\alpha}p_h$ in (\ref{3.3}), we have
\begin{eqnarray}\label{3.4}
B(u_h,p_h;u_h,\kappa_\alpha p_h)&=&
\big(u_h,\,{_{0}^{C}\!D}_{t}^{\alpha,\lambda}u_h\big)_{\Omega}+ \kappa_{\alpha}\bigg(\big(p_h,(u_h)_x\big)_{\Omega}
+ \big(p_h,p_h\big)_{\Omega}  \nonumber\\
&-&\sum_{j=1}^{N}\big[(\widehat{p}_h)_{j+\frac{1}{2}}(u_h^{-})_{j+\frac{1}{2}} - (
\widehat{p}_h)_{j-\frac{1}{2}}(u_h^{+})_{j-\frac{1}{2}}\big] \\
&+&\big(u_h ,(p_h)_x\big)_{\Omega}-\sum_{j=1}^{N}\big[(\widehat{u}_h)_{j+\frac{1}{2}} (p_h)^{-}_{j+\frac{1}{2}}
-(\widehat{u}_h)_{j-\frac{1}{2}}(p_h)^{+}_{j-\frac{1}{2}}\big]\bigg)  \nonumber\\
&=&0\nonumber.
\end{eqnarray}
Combining the numerical flux defined by (\ref{flux2.5}) and periodic boundary
conditions (\ref{Bvalue}) we arrive at (\ref{3.2}).
\end{proof}
%%%%%%%%%%%%%%%%%%%%%%%%%%%%%%%%%%%%
\begin{corollary}($L^2$-stability) The semi-discrete LDG scheme (\ref{semi-discreteLDG}) with the flux choice (\ref{flux2.5}) is $L^2$-stable, i.e.
\begin{equation}\label{stability}
 \|u_h(\cdot,t)\|^2
+2\kappa_\alpha~{_{0}I}_t^{\alpha,\lambda}\big( e^{\lambda t} \|p_h(\cdot,t)\|^2\big)
\leq e^{-\lambda t}\|u_h(\cdot,0)\|^2.
\end{equation}
\end{corollary}
%%%%%%%%%%%%%%%%%%%%%%%%%%%%%%%%%%
\begin{proof}
Using the inequality (\ref{3.1inportantiq}), we can obtain
\begin{equation}\label{important2}
e^{-\lambda t}\frac{1}{2}\; _{0}^{C}D_{t}^{\alpha,\lambda} \big\|u_h(\cdot,t)\big\|^2
+\kappa_{\alpha}\|p_h(\cdot,t)\|^2\leq 0.
\end{equation}
Recall the composite properties of the Caputo tempered fractional derivative and the Riemann-Liouville tempered fractional integral, we have
\begin{equation}\label{3.6}
_{0}I_t^{\alpha,\lambda}[_{0}^{C}D_t^{\alpha,\lambda} u(t)]=u(t)-e^{-\lambda t} u(0).
\end{equation}
Applying the operator $_{0}I_t^{\alpha,\lambda}$ on both sides of the
inequality (\ref{important2}) leads to (\ref{stability}).
\end{proof}
%%%%%%%%%%%%%%%%%%%%%%%%%%%%%%%%%%%%%%%%
%%%%%%%%%%%%%%%%%%%%%%%%%%%%%%%%%%%%%%%%
\subsection{Convergence analysis of the semi-discrete LDG scheme}\label{sec:3.2}
Now, we given the $L^2$ error estimate.
In order to give more detailed error estimate, the following two special projections operators introduced in \cite{Cockburn:00} will be used.
\begin{equation}\label{Projection1}
 \left(\mathcal {P^-}w(x)-w(x)),v(x)\right)_{I_j}=0,\forall v\in P^{k-1}(I_j),
 \mathcal {P^-}w(x^{-}_{j+1/2})=w(x_{j+1/2}),
\end{equation}
\begin{equation}\label{Projection2}
 \left(\mathcal {P^+}w(x)-w(x))v(x)\right)_{I_j}=0,\forall v\in P^{k-1}(I_j),
 \mathcal  {P^+}w(x^{+}_{j-1/2})=w(x_{j-1/2}).
\end{equation}
%%%%%%%%%%%%%%%%%%%%%%%%%%%%%%%%%%%%%%%%%%
\begin{lemma}\label{Projectionrate}\cite{Cockburn:98}
For projection operators $\mathcal{P^\pm}$, the following estimate holds
\begin{equation}\label{Prorate}
 \|w^e\| + h\|w^e\|_{\infty} + h^{\frac{1}{2}}\|w^e\|_{\Gamma_h} \leq Ch^{k + 1},
\end{equation}
where $w^e=\mathcal{P^\pm}w-w,$
 $C$ is a positive constant depending $u$ and its derivatives but independent of $h$.
\end{lemma}
%%%%%%%%%%%%%%%%%%%%%%%%%%%%%%%%
 \begin{lemma}\cite{Dixo:86}\label{gronwalla}
Let $u(t)$ be continuous and non-negative on $[0,T]$. If
\begin{equation}\label{conditiona}
u(t)\leq \varphi(t)+M\int_{0}^{t}u(s)(t-s)^{\beta}ds,0\leq t\leq T,
\end{equation}
where $0\leq\alpha<1$. $\varphi(t)$ is nonnegative monotonic increasing continuous function on $[0,T]$, and $M$ is a positive constant, then
\begin{equation}\label{conlo}
u(t)\leq \varphi(t) E_{1-\alpha}(M\Gamma(1-\alpha)t^{1-\alpha}),~
0\leq t<T,
\end{equation}
where $E_{1-\alpha}(z)$ denotes the Mittag-Leffler function defined for all $0\leq\alpha<1$ by \cite{Podlubny:99}
\begin{equation}\label{defM-L}
E_{1-\alpha}(z)=\sum_{k=0}^{\infty}\frac{z^k}{\Gamma((1-\alpha) k+1)}.
\end{equation}
\end{lemma}
%%%%%%%%%%%%%%%%%%%%%%%%%%%%
%%%%%%%%%%%%%%%%%%%%%%%%%%%%%%%%%%%
\begin{lemma} Let $u_h,p_h$ to be the
solution of semi-discrete LDG scheme (\ref{semi-discreteLDG}), and $u,p$ be the
exact solution of  (\ref{1.1}) with initial condition (\ref{intinalvalue}) and the periodic boundary
(\ref{Bvalue}), the following error estimate holds
\begin{equation}\label{3.9}
 \|u(\cdot,t)-u_h(\cdot,t)\|\leq Ch^{k+1}.
\end{equation}
\end{lemma}
%%%%%%%%%%%%%%%%%%%%%%%%%%%%%%%%%%%
\begin{proof} With the denote in (\ref{3.3}), we directly get
\begin{equation}\label{3.10}
 B(u_h,p_h;v_h,w_h)=0,~~~~\forall v_h,w_h\in V_h,
\end{equation}and
\begin{equation}\label{3.11}
 B(u,p;v_h,w_h)=0,~~~~\forall v_h,w_h\in V_h,
\end{equation}
Subtracting  (\ref{3.11}) from (\ref{3.10}), then we obtain the error equation
\begin{equation}\label{3.12}
 B(e,\overline{e};v_h,w_h)=0,~~~~\forall v_h,w_h\in V_h,
\end{equation}
where we denote $e=u-u_h,\overline{e}=p-p_h$.
We divide the error both $e$ and $\overline{e}$ into two parts
\begin{equation}\label{3.13}
\begin{array}{l}
e=u-u_{h}=(u-\mathcal{P^-}u)+(\mathcal{P^-}u-u_{h})=\varepsilon_{h}+e_{h},\\
\overline{e}=p-p_{h}=(p-\mathcal{P^+}p)+(\mathcal{P^+}p-p_{h})=\overline{\varepsilon}_{h}+\overline{e}_{h}.
\end{array}
\end{equation}
If we take $v_h=e_h,w_h=\kappa_\alpha\overline{e}_h$ in (\ref{3.12}), we get
\begin{equation}\label{3.14}
 B(e_h,\overline{e}_h;e_h,\kappa_\alpha\overline{e}_h)=-B(\varepsilon_{h},\overline{\varepsilon}_h;
 e_h,\kappa_\alpha\overline{e}_h).
\end{equation}
For the left side of (\ref{3.14}), using the equation (\ref{3.2}) in Lemma \ref{lemma2}, we have
\begin{equation}\label{3.15}
B(e_h,\overline{e}_h;e_h,\kappa_\alpha\overline{e}_h)=
\big(e_h,\,{_{0}^{C}\!D}_{t}^{\alpha,\lambda}e_h\big)_{\Omega}
 +\kappa_{\alpha}\big(\overline{e},\overline{e}\big)_{\Omega}=0.
\end{equation}
Obviously, the right of (\ref{3.14}) can be written as
\begin{eqnarray}\label{3.16}
-B(\varepsilon_{h},\overline{\varepsilon}_h;e_h,\kappa_\alpha\overline{e}_h)
:&=&-\big(e_h{_{0}^{C}\!D}_{t}^{\alpha,\lambda},\varepsilon_{h}\big)_{\Omega}
 - \kappa_{\alpha}\big(\big(\overline{\varepsilon}_h,(e_h)_x\big)_{\Omega}   \nonumber\\
&-&\sum_{j=1}^{N}\big[(\overline{\varepsilon}_h)^+_{j+\frac{1}{2}}e^{-}_{j+\frac{1}{2}}
-(\overline{\varepsilon}_h)^+_{j-\frac{1}{2}}(v_h)^{+}_{j-\frac{1}{2}}\big] \nonumber\\
&+& \big(\varepsilon_{h}, (\overline{e}_h)_x\big)_{\Omega}+ \big(\overline{\varepsilon}_h, \overline{e}_h\big)_{\Omega}   \\
&-&\sum_{j=1}^{N}\big[(\varepsilon_{h})^-_{j+\frac{1}{2}} (\overline{e}_h)^{-}_{j+\frac{1}{2}}
-(\varepsilon_{h})^-_{j-\frac{1}{2}}(\overline{e}_h)^{+}_{j-\frac{1}{2}}\big].\nonumber
\end{eqnarray}
Since $(e_h)_x$ and $(\overline{e}_h)_x$ are polynomials of degree at most $k-1$, applying the properties (\ref{Projection1})
and (\ref{Projection2}) of the projections $\mathcal {P^\pm}$, we obtain
\begin{equation*}
\big(\overline{\varepsilon}_h,(e_h)_x\big)_{I_j}=0 \quad \textrm{and} \quad \big(\varepsilon_{h}, (\overline{e}_h)_x\big)_{I_j}=0.
\end{equation*}
In other way,
\begin{equation*}
({\varepsilon}_h)^{-}_{j+\frac{1}{2}}=u_{j+\frac{1}{2}}-(\mathcal {P}u)^{-}_{j+\frac{1}{2}}=0 \quad \textrm{and} \quad ({\overline{\varepsilon}}_h)^{+}_{j-\frac{1}{2}}=p_{j-\frac{1}{2}}-(\mathcal {P}p)^{+}_{j-\frac{1}{2}}
=0.
\end{equation*}
By the Cauchy's inequality, we get
\begin{eqnarray}\label{3.17}
\big(e_h,\,{_{0}^{C}\!D}_{t}^{\alpha,\lambda}e_h\big)_{\Omega}
 +\kappa_{\alpha}\big(\overline{e},\overline{e}\big)_{\Omega}
&\leq&\frac{1}{2}
\big({_{0}^{C}\!D}_{t}^{\alpha,\lambda}\varepsilon_{h}
,{_{0}^{C}\!D}_{t}^{\alpha,\lambda}\varepsilon_{h}\big)_{\Omega}
+\frac{1}{2}\big(e_h,e_h\big)_{\Omega}  \nonumber \\
&+&\frac{1}{2}\kappa_{\alpha}\big(\overline{\varepsilon}_h,\overline{\varepsilon}_h\big)_{\Omega}
+\frac{1}{2}\kappa_{\alpha}\big(\overline{e}_h,\overline{e}_h\big)_{\Omega}.
\end{eqnarray}
From the Lemma \ref{Projectionrate}, we conclude that
\begin{equation}\label{3.18}
2\big(e_h\,{_{0}^{C}\!D}_{t}^{\alpha,\lambda},e_h\big)_{\Omega}
 +\kappa_{\alpha}\big(\overline{e},\overline{e}\big)_{\Omega}
\leq \big(e_h,e_h\big)_{\Omega}+Ch^{2k+2}.
\end{equation}
Using the inequality (\ref{3.1inportantiq}), we have
\begin{eqnarray*}\label{3.19}
 {_{0}^{C}\!D}_{t}^{\alpha,\lambda}\|e_h(\cdot,t)\|^2
+\kappa_{\alpha}e^{\lambda t}\|\overline{e}(\cdot,t)\|^2
\leq e^{\lambda t}\|e_h(\cdot,t)\|^2+e^{\lambda t}Ch^{2k+2}.
\end{eqnarray*}
Combining the composite properties (\ref{3.6}) and the definition of
Riemann-Liouville tempered fractional integral, we arrive at
\begin{eqnarray}\label{3.20}
e^{2\lambda t}\|e_h(\cdot,t)\|^2&\leq&\frac{e^{\lambda t}}
{\Gamma(\alpha)}\int_0^t(t-\tau)^{\alpha-1}
\big(e^{\lambda \tau}\|e_h(\cdot,\tau)\|\big)^2d\tau \nonumber\\
&&+Ch^{2k+2}\frac{e^{\lambda t}}
{\Gamma(\alpha)}\int_0^t(t-\tau)^{\alpha-1}e^{2\lambda \tau}d\tau, \\
&\leq&\frac{e^{\lambda t}}
{\Gamma(\alpha)}\int_0^t(t-\tau)^{\alpha-1}
\big(e^{\lambda \tau}\|e_h(\cdot,\tau)\|\big)^2d\tau
+ C e^{3\lambda t} h^{2k+2},  \nonumber
\end{eqnarray}
where we used the fact
\begin{eqnarray}\label{argu3.20}
\int_0^t(t-\tau)^{\alpha-1}e^{2\lambda \tau}d\tau=
e^{2\lambda t}\int_0^ts^{\alpha-1}e^{-2\lambda s}ds
\leq
 \frac{e^{2\lambda t}}{(2\lambda)^\alpha}\Gamma(\alpha).
%&=&
% \frac{e^{2\lambda t}}{(2\lambda)^\alpha}\int_0^{+\infty}y^{\alpha-1}e^{-y}dy
\end{eqnarray}
Furthermore, using the fractional Gronwall's lemma \ref{gronwalla}, we have
\begin{equation}\label{3.21}
\|e_h(\cdot,t)\|^2 \leq Ce^{\lambda T}E_{\alpha}(e^{\lambda t}t^{\alpha})h^{2k+2},
\end{equation}
where $E_{\alpha}(\cdot)$ denotes the Mittag-Leffler function is defined by (\ref{defM-L}).
\end{proof}
%%%%%%%%%%%%%%%%%%%%%%%%%%%%%%%
%%%%%%%%%%%%%%%%%%%%%%%%%%%%%%%
\section{Fully discrete LDG schemes}\label{sec:4}
In this section we discrete the time in the semi-discrete scheme by virtue of high
order  approximation. Let $0= t_0 < t_1 <\cdots <t_n < t_{n+1} <\cdots< t_M = T$ be the subdivision of the time
interval $[0,T ]$, with the time step $\tau= t_{n+1}-t_n$.
To achieve the high order accuracy, we employ the $q$-th order approximations given in \cite{Chen:14} to approximate the Riemann-Liouville tempered derivative, i.e.
\begin{equation}\label{RLTCH}
{_0D}_t^{\alpha,\lambda} v(t)|_{t_n}=\tau^{-\alpha}\sum_{k=0}^{n}{d}_{k}^{q,\alpha}v(t_{n-k})+R^n, ~~q=1,2,3,4,5,
\end{equation}
where  $R^n=\mathcal{O}(\tau^q)$ and
$${d}_{k}^{q,\alpha}=e^{-\lambda k \tau}{l}_k^{q,\alpha},~~q=1,2,3,4,5.$$
More details  of ${l}_k^{q,\alpha}$, one can refer to \cite{Chen:14}. Using (\ref{RLTCH}), we find
\begin{equation}\label{4.3}
\begin{split}
&{_0D}_t^{\alpha,\lambda} u(x,t)|_{(x_i,t_n)}=\tau^{-\alpha}\sum_{k=0}^{n}{d}_{k}^{q,\alpha}u(x_i,t_{n-k})+\mathcal{O}(\tau^q), \\
&{_0D}_t^{\alpha,\lambda} [e^{-\lambda t}u(x,0)]_{(x_i,t_n)}=\tau^{-\alpha}\sum_{k=0}^{n}{d}_{k}^{q,\alpha}e^{-\lambda (n-k)\tau}u(x_i,0)+\mathcal{O}(\tau^q).
\end{split}
\end{equation}
Recalling the relation of Riemann-Liouville and Caputo tempered fractional derivatives \cite{Sabzikar:15,Li:14}
\begin{equation}\label{4.1}
{^C_0}\!{D}_t^{\alpha,\lambda}v( t)
={_0D}_t^{\alpha,\lambda} [v( t)-e^{-\lambda t}v( 0)],
\end{equation}
the weak form of the first order system (\ref{2.3}) at $t_n$ can be rewritten as
\begin{equation}\label{weakform}
\begin{array}{l}
\displaystyle
\tau^{-\alpha}\sum_{k=0}^{n}{d_{k}^\alpha}\big(u(x,t_{n-k}), v\big)_{\Omega}
-\tau^{-\alpha}\sum_{k=0}^{n}{d_{k}^{\alpha}}e^{-\lambda (n-k)\tau}\big(u(x,t_0), v\big)_{\Omega} +\kappa_{\alpha}\big(p(x,t_{n}),v_x\big)_{\Omega} \\
\displaystyle
\quad\quad\quad\quad\quad\quad\quad\quad-\kappa_{\alpha}\sum_{j=1}^N\big[(p(x,t_{n})v^{-})_{j+\frac{1}{2}}- (p(x,t_{n})v^{+})_{j-\frac{1}{2}}\big]
=\big(T^n,v\big)_{\Omega} \\
\displaystyle
\big(p(x,t_{n}), w\big)_{\Omega}+\big(u(x,t_{n}), w_x\big)_{\Omega}-\sum_{j=1}^N\big[(u(x,t_{n})
w^{-})_{j+\frac{1}{2}}
-(u(x,t_{n}) w^{+})_{j-\frac{1}{2}}\big]=0.
\end{array}
\end{equation}
Let $u_h^n,p_h^n \in V_h$ be the approximate solution of $u(x,t_n),p(x,t_n)$, respectively. We propose the fully discrete LDG schemes as follows: Find $u_h^n,p_h^n\in V_h$,
\begin{equation}\label{fullyLDG}
\begin{array}{l}
\displaystyle
\big(p_h^n, w\big)_{\Omega} + \big(u_h^n, w_x\big)_{\Omega} - \sum_{j=1}^{N}\big[(\widehat{u}^n_h
w^{-})_{j+\frac{1}{2}} - (\widehat{u}_h^n w^{+})_{j-\frac{1}{2}}\big]=0,  \\
\displaystyle
 {l}_{0}^{q,\alpha}\big( u_h^n, v\big)_{\Omega} + \kappa_{\alpha}\tau^{\alpha}\bigg(\big(p_h^n, v_x\big)_{\Omega}
-\sum_{j=1}^{N}\big[(\widehat{p}^n_h
v^{-})_{j+\frac{1}{2}} - (\widehat{p}_h^n v^{+})_{j-\frac{1}{2}}\big]\bigg)      \\
\displaystyle
=e^{-\lambda n \tau}\sum_{k=0}^{n-1}{l}_k^{q,\alpha}\big(u_h^0, v\big)_{\Omega}
-\sum_{k=1}^{n-1}e^{-\lambda k \tau}{l}_k^{q,\alpha}\big(u_h^{n-k}, v\big)_{\Omega},
\end{array}
\end{equation}
for all $v,w\in V_h$, $j=1,2,...,N$.
We take numerical flux to be $\widehat{u}_h^n=(u_h^n)^{-},\widehat{q}_h^n=(q_h^n)^{+}$ with the same choice of (\ref{flux2.5}).
In the following, we prove the stability and error estimate of the schemes (\ref{fullyLDG}) with $q=1$ in $L_2$ norm.
For convenience, we denote ${l}_k^{1,\alpha}$ by $w_k$,
where the coefficients
\begin{equation}\label{4.5}
w_k=(-1)^k\left ( \begin{matrix} \alpha \\ k\end{matrix} \right ),
w_0=1, ~~w_k=\left(1-\frac{\alpha+1}{k}\right)w_{k-1},~~k \geq {1}.
\end{equation}
%%%%%%%%%%%%%%%%%%%%%%%%%%%%%
\begin{lemma}\cite{Liu:15}\label{lemma2.1}
The coefficients $w_k$ defined in (\ref{4.5}) satisfy
\begin{equation}\label{4.6}
  w_0=1; ~w_k<0,~k\geq 1; ~~\sum_{k=0}^{n-1}w_k>0;  ~\sum_{k=0}^{\infty}w_k=0;
\end{equation}
and
\begin{equation}\label{4.7}
 \frac{1}{n^\alpha \Gamma(1-\alpha)}< \sum_{k=0}^{n-1}w_k=-\sum_{k=n}^{\infty}w_k \leq \frac{1}{n^\alpha}, ~{\rm for}~n\geq 1.
\end{equation}
\end{lemma}
%%%%%%%%%%%%%%%%%%%%%%%%%%%%%%%%
\begin{theorem}The fully discrete LDG scheme (\ref{fullyLDG}) of   initial-boundary
problem (\ref{1.1})-(\ref{Bvalue}) is
unconditional stability and holds
\begin{equation}\label{fullystab}
\|u_h^n\|\leq \|u_h^0\|,~n\geq1.
\end{equation}
\end{theorem}
\begin{proof}
Setting $v = u_h^n,w=\kappa_{\alpha}\tau^\alpha p_h^n$ in (\ref{fullyLDG}) and summing over all elements, we obtain
\begin{equation}\label{4.10}
\begin{array}{l}
\displaystyle
\|u_h^n\|^2+\kappa_{\alpha}\tau^{\alpha}\|p^n_{h}\|^2+\kappa_{\alpha}\tau^\alpha\sum_{j=1}^{N}
\big[F_{j+1/2}(u_h^n,p_h^n)-F_{j-1/2}(u_h^n,p_h^n)+\Theta_{j-1/2}(u_h^n,p_h^n)\big]   \\
\displaystyle
=e^{-\lambda n \tau}\sum_{k=0}^{n-1}w_k\big(u_h^0, u_h^n\big)_{\Omega}
-\sum_{k=1}^{n-1}e^{-\lambda k \tau}w_k\big(u_h^{n-k}, u_h^n\big)_{\Omega},
\end{array}
\end{equation}
where $$F(u_h^n,p_h^n)=(p_h^n)^{-}(u^n_h)^{-} - (\widehat{u}_{h}^{n})(p_h^n)^{-} - (u_h^n)^{-}(\widehat{p}_{h}^{n}),$$
and
\begin{equation}\label{theta}
\begin{split}
\Theta(u_h^n,p_h^n)&= (p_h^n)^{-}(u_h^n)^{-} + (\widehat{u}_h^n)(p_h^n)^{+} + (u_h^n)^{+}(\widehat{p}_h^n)           \\
&\quad- (u_h^n)^{+}(p_h^n)^{+} - (\widehat{u}_h^n)(p_h^n)^{-} - (u_h^n)^{-}(\widehat{p}_h^n).
\end{split}
\end{equation}
Recalling  the numerical flux in (\ref{flux2.5}), we have $\Theta(u_h^n,p_h^n)=0$.
On the other hand, in view of the periodic boundary conditions (\ref{Bvalue}), we get
$$\sum_{j=1}^{N}\big[F_{j+1/2}(u_h^n,p_h^n)-F_{j-1/2}(u_h^n,p_h^n)]
=F_{N+1/2}(u_h^n,p_h^n)-F_{1/2}(u_h^n,p_h^n)=0.$$
Using the Cauchy-Schwartz inequality, we arrive at
\begin{equation}\label{4.11}
\begin{array}{l}
\displaystyle
\|u_h^n\|^2+\kappa_{\alpha}\tau^{\alpha}\|p_h^n\|^2 \leq e^{-\lambda n \tau}\sum_{k=0}^{n-1}{w_k}\|u_h^0\| \|u_h^n\|
-\sum_{k=1}^{n-1}{w_k}e^{-\lambda k \tau}\|u_h^n\|\| u_h^{n-k}\|.
\end{array}
\end{equation}
Therefore, we have
\begin{equation}\label{4.12}
\begin{split}
\|u_h^n\| &\leq e^{-\lambda n \tau}\sum_{k=0}^{n-1}{w_k}\|u_h^0\|
-\sum_{k=1}^{n-1}e^{-\lambda k \tau}{w_k}\|u_h^{n-k}\|           \\
&\leq \sum_{k=0}^{n-1}{w_k}\|u_h^0\| - \sum_{k=1}^{n-1}{w_k}\|u_h^{n-k}\|.
\end{split}
\end{equation}
Next we need to prove the following estimate by mathematical induction
\begin{equation}\label{stable}
\|u_h^n\| \leq \|u_h^0\|.
\end{equation} From the inequality (\ref{4.12}), we can see the
inequality (\ref{stable}) holds obviously when $n=1$. Assuming
\begin{equation*}
\|u_h^m\| \leq \|u_h^0\|, ~~{\rm for}~~m=1,2,\ldots,n-1,
\end{equation*}
then from the inequality (\ref{4.12}), we obtain
\begin{equation*}
\|u_h^n\|
\leq \sum_{k=0}^{n-1} w_k\|u_h^{0}\|- \sum_{k=1}^{n-1}w_k\|u_h^{n-k}\|
\leq \sum_{k=0}^{n-1} w_k\|u_h^{0}\|- \sum_{k=1}^{n-1}w_k\|u_h^{0}\|=\|u_h^{0}\|.
\end{equation*}
The proof is complete.
\end{proof}
%%%%%%%%%%%%%%%%%%%%%%
\begin{theorem}\label{theoremfullerr2} Let $u(x,t_n)$ be the exact solution of the
problem (\ref{1.1})-(\ref{Bvalue}), which is sufficiently smooth such that
$u\in H^{m+1}$ with $0\leq m\leq k+1$. Let $u_h^n$ be the numerical
solution of the fully discrete LDG scheme (\ref{fullyLDG}), then there
holds the following error estimate
\begin{equation}\label{errorestimate}
\begin{split}
\|u(x,t_n)-u_h^n\|\leq C(\tau+h^{k+1}), ~~n=1,\cdots,M,
\end{split}
\end{equation}
where $C$ is a constant depending on $u, T, \alpha$ but independent of $\tau$ and $h$.
\end{theorem}
\begin{proof}
To simplify the notation, we decompose the errors as follows:
\begin{equation}\label{errnotation}
\begin{split}
&e_u^n=u(x,t_n)-\mathcal{P^-}u(x,t_n)+\mathcal{P^-}u(x,t_n)-u_h^n=\mathcal{P^-}e_u^n-\mathcal{P^-}\varepsilon_u^n,   \\
&e_p^n=p(x,t_n)-\mathcal{P^+}p(x,t_n)+\mathcal{P^+}p(x,t_n)-p_h^n=\mathcal{P^+}e_p^n-\mathcal{P^+}\varepsilon_p^n.
\end{split}
\end{equation}
Combining (\ref{weakform}) and (\ref{fullyLDG}), we have
\begin{equation}\label{4.13}
\begin{array}{l}
\begin{split}
&w_{0}\big(e_u^n, v\big)_{\Omega} + \kappa_{\alpha}\tau^{\alpha}\bigg(\big(e_p^n, v_x\big)_{\Omega}
-\sum_{j=1}^{N}\big[((e^n_p)^{+}
v^{-})_{j+\frac{1}{2}} - ((e_p^n)^{+} v^{+})_{j-\frac{1}{2}}\big]\bigg)       \\
&+\big(e_p^n, w\big)_{\Omega} + \big(e_u^n, w_x\big)_{\Omega} - \sum_{j=1}^{N}\big[((e^n_u)^{-}
w^{-})_{j+\frac{1}{2}} - ((e_u^n)^{-} w^{+})_{j-\frac{1}{2}}\big]  \\
& = e^{-\lambda n \tau}\sum_{k=0}^{n-1}w_k\big(e_u^0, v\big)_{\Omega}
-\sum_{k=1}^{n-1}e^{-\lambda k \tau}w_k\big(e_u^{n-k}, v\big)_{\Omega}
-\tau^{\alpha}\big(R^{n}, v\big)_{\Omega}.
\end{split}
\end{array}
\end{equation}
Substituting (\ref{errnotation}) into (\ref{4.13}) and notice that $e_u^0=0$, we can get the error equation
\begin{equation}\label{4.14}
\begin{array}{l}
\displaystyle
w_{0}\big(\mathcal{P^-}e_u^n, v\big)_{\Omega} + \kappa_{\alpha}\tau^{\alpha}\left(\big(\mathcal{P^+}e_p^n, v_x\big)_{\Omega}
-\sum_{j=1}^{N}\big[((\mathcal{P^+}e^n_p)^{+}
v^{-})_{j+\frac{1}{2}} - ((\mathcal{P^+}e_p^n)^{+} v^{+})_{j-\frac{1}{2}}\big]\right)  \\
\displaystyle
+\big(\mathcal{P^+}e_p^n, w\big)_{\Omega} + \big(\mathcal{P^-}e_u^n, w_x\big)_{\Omega}
- \sum_{j=1}^{N}\big[((\mathcal{P^-}e^n_u)^{-}
w^{-})_{j+\frac{1}{2}} - ((\mathcal{P^-}e_u^n)^{-} w^{+})_{j-\frac{1}{2}}\big]  \\
\displaystyle
=-\sum_{k=1}^{n-1}e^{-\lambda k \tau}w_k\big(\mathcal{P^-}e_u^{n-k}, v\big)_{\Omega}
-\tau^{\alpha}\big(R^{n}, v\big)_{\Omega}
\quad +w_{0}\big(\mathcal{P^-}\varepsilon_{p}^{n},v\big)_{\Omega}   \\
\displaystyle
\quad+ \kappa_{\alpha}\tau^{\alpha}\big(\mathcal{P^+}\varepsilon_{p}^{n}, v_x\big)_{\Omega}
-\kappa_{\alpha}\tau^{\alpha}\sum_{j=1}^{N}\big[\big((\mathcal{P^+}\varepsilon_{u}^{n})^{+}
v^{-}\big)_{j+\frac{1}{2}}
- \big((\mathcal{P^+}\varepsilon_{p}^{n})^{+} v^{+}\big)_{j-\frac{1}{2}}\big]
 \\
\displaystyle
\quad+\big(\mathcal{P^+}\varepsilon_{p}^{n}, w\big)_{\Omega}  - \sum_{j=1}^{N}\big[\big((\mathcal{P^-}\varepsilon_{u}^{n})^{-}
w^{-}\big)_{j+\frac{1}{2}} - \big((\mathcal{P^-}\varepsilon_{p}^{n})^{-} w^{+}\big)_{j-\frac{1}{2}}\big]   \\
\displaystyle
\quad+ \sum_{k=1}^{n-1}e^{-\lambda k \tau}w_k\big(\mathcal{P^-}\varepsilon_{u}^{n-k}, v\big)_{\Omega}
+ \big(\mathcal{P^-}\varepsilon_{u}^{n}, w_x\big)_{\Omega}.
\end{array}
\end{equation}
Taking $v=\mathcal{P^-}e_u^n$, $w=\kappa_\alpha\tau^\alpha\mathcal{P^+}e_p^n$ in (\ref{4.14}), we get
\begin{equation}\label{4.15}
\begin{array}{l}
\displaystyle
w_{0}\big(\mathcal{P^-}e_u^n,\mathcal{P^-}e_u^n \big)_{\Omega} +
\kappa_{\alpha}\tau^{\alpha}\big(\mathcal{P^+}e_p^n,\mathcal{P^+}e_p^n \big)_{\Omega}  \\
\displaystyle
=-\sum_{k=1}^{n-1}e^{-\lambda k \tau}w_k\big(\mathcal{P^-}e_u^{n-k}, \mathcal{P^-}e_u^n\big)_{\Omega}
    \\
\displaystyle
-\tau^{\alpha}\big(R^{n}, \mathcal{P^-}e_u^n\big)_{\Omega} + w_{0}\big(\mathcal{P^-}\varepsilon_u^n,\mathcal{P^-}e_u^n\big)_{\Omega}
+ \kappa_{\alpha}\tau^{\alpha}\big(\mathcal{P^+}\varepsilon_p^m, (\mathcal{P^-}e_u^n)_x\big)_{\Omega}     \\
\displaystyle
-\kappa_{\alpha}\tau^{\alpha}\sum_{j=1}^{N}\big[\big((\mathcal{P^+}\varepsilon_p^n)^{+}
(\mathcal{P^-}e_u^n)^{-}\big)_{j+\frac{1}{2}} - \big((\mathcal{P^+}\varepsilon_p^n)^{+}
(\mathcal{P^-}e_u^n)^{+}\big)_{j-\frac{1}{2}}\big]         \\
\displaystyle
+\kappa_\alpha\tau^\alpha\big(\mathcal{P^+}\varepsilon_p^n, \mathcal{P^+}e_p^n\big)_{\Omega}
+ \kappa_\alpha\tau^\alpha \big(\mathcal{P^-}\varepsilon_u^n, (\mathcal{P^+}e_p^n)_x\big)_{\Omega}   \\
\displaystyle
- \kappa_\alpha\tau^\alpha \sum_{j=1}^{N}\big[\big((\mathcal{P^-}\varepsilon_u^n)^{-}
(\mathcal{P^+}e_p^n)^{-}\big)_{j+\frac{1}{2}} - \big((\mathcal{P^-}\varepsilon_u^n)^{-}
(\mathcal{P^+}e_p^n)^{+}\big)_{j-\frac{1}{2}}\big]      \\
\displaystyle
+ \sum_{k=1}^{n-1}e^{-\lambda k \tau}w_k\big(\mathcal{P^-}\varepsilon_u^{n-k}, \mathcal{P^-}e_u^n\big)_{\Omega}.
\end{array}
\end{equation}
Using the properties of projections $\mathcal{P^\pm}$ and $w_0=1$, we can further get
\begin{equation}\label{4.16}
\begin{array}{l}
\displaystyle
\|(\mathcal{P^-}e_u^n)\|^2 + \kappa_{\alpha}\tau^{\alpha}\|(\mathcal{P^+}e_p^n)\|^2 \\
\displaystyle
=-\sum_{k=1}^{n-1}e^{-\lambda k \tau}w_k\big(\mathcal{P^-}e_u^{n-k}, \mathcal{P^-}e_u^n\big)_{\Omega}
-\tau^{\alpha}\big(R^{n}, \mathcal{P^-}e_u^n\big)_{\Omega}    \\
\displaystyle
+\big(\mathcal{P^-}\varepsilon_u^n,\mathcal{P^-}e_u^n\big)_{\Omega}
+\kappa_\alpha\tau^\alpha\big(\mathcal{P^+}\varepsilon_p^n, \mathcal{P^+}e_p^n\big)_{\Omega}   \\
\displaystyle
+ \sum_{k=1}^{n-1}e^{-\lambda k \tau}w_k\big(\mathcal{P^-}\varepsilon_u^{n-k}, \mathcal{P^-}e_u^n\big)_{\Omega}.
\end{array}
\end{equation}
Applying the Cauchy-Schwarz inequality, we have
\begin{equation}\label{4.17}
\begin{array}{l}
\displaystyle
\|(\mathcal{P^-}e_u^n)\|^2  + \kappa_{\alpha}\tau^{\alpha}\|(\mathcal{P^+}e_p^n)\|^2  \\
\displaystyle
\leq -\sum_{k=1}^{n-1}e^{-\lambda k \tau}w_k\|\mathcal{P^-}e_u^{n-k}\| ||\mathcal{P^-}e_u^n||
+\tau^{\alpha}\|R^{n}\| \|\mathcal{P^-}e_u^n\|    \\
\displaystyle
+\|\mathcal{P^-}\varepsilon_u^n\| \|\mathcal{P^-}e_u^n\|
+\kappa_\alpha\tau^\alpha\|\mathcal{P^+}\varepsilon_p^n\| \|\mathcal{P^+}e_p^n\|   \\
\displaystyle
- \sum_{k=1}^{n-1}e^{-\lambda k \tau}w_k \|\mathcal{P^-}\varepsilon_{u}^{n-k}\| \|\mathcal{P^-}e_u^n\|.
\end{array}
\end{equation}
Combining the inequality $2ab\leq a^2 + b^2$ and $e^{-\lambda k \tau}\in (0,1]$,
from the above inequality (\ref{4.17}), we can derive
\begin{equation}\label{4.18}
\begin{array}{l}
\displaystyle
\|(\mathcal{P^-}e_u^n)\|^2  + \kappa_{\alpha}\tau^{\alpha}\|(\mathcal{P^+}e_p^n)\|^2
\\
\displaystyle
\leq \frac{1}{2} \big(-\sum_{k=1}^{n-1}w_k\|\mathcal{P^-}e_u^{n-k}\|
+\tau^{\alpha}\|R^{n}\|+\|\mathcal{P^-}\varepsilon_u^n\|   \\
\displaystyle
- \sum_{k=1}^{n-1}w_k \|\mathcal{P^-}\varepsilon_u^{n-k}\| \big)^2
+\frac{1}{2} \|\mathcal{P^-}e_u^n\|^2   \\
\displaystyle
+\frac{1}{2} \kappa_\alpha\tau^\alpha\|\mathcal{P^+}\varepsilon_p^n\|^2
+\frac{1}{2} \kappa_\alpha\tau^\alpha\|\mathcal{P^+}e_p^n\|^2.
\end{array}
\end{equation}
Moreover, we have
\begin{equation}\label{4.19}
\begin{array}{l}
\displaystyle
\|(\mathcal{P^-}e_u^n)\|
\leq -\sum_{k=1}^{n-1}w_k\|\mathcal{P^-}e_u^{n-k}\|
+\tau^{\alpha}\|R^{n}\|     \\
\displaystyle
\quad\quad\quad\quad+\|\mathcal{P^-}\varepsilon_u^n\|- \sum_{k=1}^{n-1}w_k \|\mathcal{P^-}\varepsilon_u^{n-k}\|
+\sqrt{\kappa_\alpha\tau^\alpha} \|\mathcal{P^+}\varepsilon_p^n\|.
\end{array}
\end{equation}
Next, we prove the following estimate by mathematical introduction
\begin{equation}\label{4.20}
\|(\mathcal{P^-}e_u^n)\| \leq C(\tau+h^{k+1}).
\end{equation}
For $n=1$, using the properties (\ref{Prorate}) of the projections $\mathcal{P^\pm}$, it can be seen that
the inequality (\ref{4.20}) holds obviously. Assuming
$$\|(\mathcal{P^-}e_u^m)\| \leq C(\tau+h^{k+1}), ~~{\rm for} ~m=1,2\cdots,n-1.$$
Remembering $-\sum_{k=1}^{n-1}w_k<1 $ and the properties of the projections $\mathcal{P^\pm}$, we have
\begin{equation}\label{err}
\begin{split}
\|(\mathcal{P^-}e_u^n)\|
\leq C(\tau+h^{k+1}).
\end{split}
\end{equation}
Finally, combining the triangle inequality and lemma \ref{Projectionrate} to have
\begin{eqnarray} \nonumber
\|u(x,t_n)-u_h^n\| &=& \|\mathcal{P^-}e_u^n - \mathcal{P^-}\varepsilon_u^n\|  \\   \nonumber
\displaystyle
&\leq& \|\mathcal{P^-}e_u^n\| + \|\mathcal{P^-}\varepsilon_u^n\|   \\   \nonumber
\displaystyle
&\leq& C(\tau+h^{k+1}).
\end{eqnarray}
\end{proof}
%%%%%%%%%%%%%%%%%%%%%%%%%%%%%%%
\section{Numerical experiments}\label{sec:5}
In this section, we perform three examples to illustrate the effectiveness
of our numerical schemes and confirm  our theoretical results.
%%%%%%%%%%%%%%%%%%%%%%%%%%%%%
\begin{example}\label{example1}
Without loss of generality, we add a force term $f(x,t)$ on the right hand side of the equation (\ref{1.1}), we consider
\begin{equation}\label{5.1}
 {_{0}^{C}\!D}_{t}^{\alpha,\lambda}u(x,t)=\kappa_{\alpha}u_{xx}(x,t)+f(x,t), ~~(x,t)\in[0,1]\times(0,T],\\
\end{equation}
with periodic boundary conditions $u(x+\pi,t)=u(x,t)$ and initial condition
$u(x,0)=sin(2\pi x)$. If we take the force term $f(x,t)$ as
$f(x,t)=e^{-\lambda t}\bigg(\frac{\Gamma(\beta+1)}{\Gamma(\beta+1-\alpha)}t^{\beta-\alpha}
+4\kappa_{\alpha}\pi^2(t^{\beta}+1)\bigg)sin(2\pi x),$
then the exact solution of the problem (\ref{5.1}) with the corresponding initial-boundary condition gives
$$ u(x,t)=e^{-\lambda t}(t^{\beta}+1)sin(2\pi x).$$
\end{example}
%%%%%%%%%%%%%%%%%%%%%%%%%%%%%%%%%%%%%%%%%%%%
%%%%%%=======================================================
\begin{table}[htp]
%\centering
\caption{The $L^2$ errors and convergence orders for different $\lambda$ at $T=1$ with $\alpha=0.5$, $\tau=h^{2/3}$.}
\vspace{0.2cm}
\begin{small}
  \begin{tabular}{c|c|c|c|c|c|c|c}
\hline
&  &\multicolumn{2}{c|}{$\lambda=0.8$} & \multicolumn{2}{c|}{$\lambda=1.5$} & \multicolumn{2}{c}{$\lambda=2$}\\
\cline{3-4} \cline{5-6} \cline{7-8}
$k$& $h$ &$\|\cdot\|$-error  & order & $\|\cdot\|$-error & order & $\|\cdot\|$-error & order\\
\hline
\multirow{4}{*}{$1$}
   &1/10 & 1.757029e-02 &         &8.266055e-03 &        &4.823736e-03 &         \\
   &1/20 & 4.472764e-03 & 1.9739  &2.091684e-03 & 1.9825 &1.215416e-03 & 1.9887  \\
   &1/40 & 9.995738e-04 & 2.1618  &4.874263e-04 & 2.1014 &2.918227e-04 & 2.0583  \\
   &1/80 & 2.487744e-04 & 2.0065  &1.215344e-04 & 2.0038 &7.285848e-05 & 2.0019  \\
 \hline
  \end{tabular}
\label{tabexmone1}
\end{small}
\end{table}
%%%%%%=======================================================
%%%%%%=======================================================
\begin{table}[htp]
\centering
\caption{The $L^2$ errors and convergence orders for different $\lambda$ at $T=1$ with $\alpha=0.5$, $\tau=h^{3/2}$.}
\vspace{0.2cm}
\begin{small}
  \begin{tabular}{c|c|c|c|c|c|c|c}
\hline
&  &\multicolumn{2}{c|}{$\lambda=0.8$} & \multicolumn{2}{c|}{$\lambda=1.5$} & \multicolumn{2}{c}{$\lambda=2$}\\
\cline{3-4} \cline{5-6} \cline{7-8}
$k$& $h$ &$\|\cdot\|$-error  & order & $\|\cdot\|$-error & order & $\|\cdot\|$-error & order\\
\hline
\multirow{4}{*}{$2$}
   &1/10 & 7.817372e-04 &         &3.849712e-04 &        &2.321083e-04 &        \\
   &1/20 & 9.719738e-05 & 3.0077  &4.805674e-05 & 3.0019 &2.905722e-05 & 2.9978  \\
   &1/40 & 1.203610e-05 & 3.0135  &5.976656e-06 & 3.0073 &3.624898e-06 & 3.0029  \\
   &1/80 & 1.506415e-06 & 2.9982  &7.477283e-07 & 2.9988 &4.533749e-07 & 2.9992  \\
 \hline
  \end{tabular}
\label{tabexmone2}
\end{small}
\end{table}
%%%%%%=======================================================
\begin{table}[htp]
\centering
\caption{The $L^2$ errors and convergence orders for different $\lambda$ at $T=1$ with $\alpha=0.5$, $\tau=h^{4/3}$.}
\vspace{0.2cm}
\begin{small}
  \begin{tabular}{c|c|c|c|c|c|c|c}
\hline
&  &\multicolumn{2}{c|}{$\lambda=0.8$} & \multicolumn{2}{c|}{$\lambda=1.5$} & \multicolumn{2}{c}{$\lambda=2$}\\
\cline{3-4} \cline{5-6} \cline{7-8}
$k$ & $h$ &$\|\cdot\|$-error  & order & $\|\cdot\|$-error & order & $\|\cdot\|$-error & order\\
\hline
\multirow{4}{*}{$3$}
   &1/10 & 3.174252e-05 &         &1.553122e-05 &        &9.321072e-06 &        \\
   &1/20 & 1.971099e-06 & 4.0093  &9.698783e-07 & 4.0012 &5.844179e-07 & 3.9954  \\
   &1/40 & 1.211423e-07 & 4.0242  &6.009536e-08 & 4.0125 &3.642279e-08 & 4.0041  \\
   &1/80 & 7.570034e-09 & 4.0003  &3.756951e-09 & 3.9996 &2.277746e-09 & 3.9992  \\
 \hline
  \end{tabular}
\label{tabexmone3}
\end{small}
\end{table}
%%%%%%=======================================================
\begin{table}[htp]
\centering
\caption{The $L^2$ errors and convergence orders for different $\lambda$ at $T=1$ with $\alpha=0.5$, $\tau=h^{5/4}$.}
\vspace{0.2cm}
\begin{small}
  \begin{tabular}{c|c|c|c|c|c|c|c}
\hline
&  &\multicolumn{2}{c|}{$\lambda=0.8$} & \multicolumn{2}{c|}{$\lambda=1.5$} & \multicolumn{2}{c}{$\lambda=2$}\\
\cline{3-4} \cline{5-6} \cline{7-8}
$k$ & $h$ &$\|\cdot\|$-error  & order & $\|\cdot\|$-error & order & $\|\cdot\|$-error & order\\
\hline
\multirow{4}{*}{$4$}
   &1/5  & 4.435466e-05 &         &2.097285e-05 &        &1.228325e-05 &        \\
   &1/10 & 1.380329e-06 & 5.0060  &6.796152e-07 & 4.9477 &4.096977e-07 & 4.9060  \\
   &1/20 & 4.402038e-08 & 4.9707  &2.160624e-08 & 4.9752 &1.299606e-08 & 4.9784  \\
   &1/40 & 1.374369e-09 & 5.0013  &6.805691e-10 & 4.9886 &4.119552e-10 & 4.9794  \\
 \hline
  \end{tabular}
\label{tabexmone4}
\end{small}
\end{table}
%%%%%%=======================================================
\begin{table}[htp]
\centering
\caption{The $L^2$ errors and convergence orders for different $\lambda$ at $T=1$ with $\alpha=0.5$, $\tau=h^{3/2}$.}
\vspace{0.2cm}
\begin{small}
  \begin{tabular}{c|c|c|c|c|c|c|c}
\hline
&  &\multicolumn{2}{c|}{$\lambda=0.8$} & \multicolumn{2}{c|}{$\lambda=1.5$} & \multicolumn{2}{c}{$\lambda=2$}\\
\cline{3-4} \cline{5-6} \cline{7-8}
$k$ & $h$ &$\|\cdot\|$-error  & order & $\|\cdot\|$-error & order & $\|\cdot\|$-error & order\\
\hline
\multirow{4}{*}{$5$}
   &1/5  & 6.414671e-06 &         &3.026083e-06 &        &1.769351e-06 &        \\
   &1/10 & 1.061861e-07 & 5.9167  &5.229200e-08 & 5.8547 &3.152810e-08 & 5.8104  \\
   &1/20 & 1.689493e-09 & 5.9739  &8.353260e-10 & 5.9681 &5.050749e-10 & 5.9640  \\
   &1/40 & 2.658552e-11 & 5.9898  &1.320140e-11 & 5.9836 &8.006770e-12 & 5.9791  \\
 \hline
  \end{tabular}
\label{tabexmone5}
\end{small}
\end{table}
%%%%%=============================================
%%%%%%=======================================================
\begin{table}[htp]
\centering
\caption{The $L^2$ errors and convergence orders for different $\alpha$, $\lambda$ at $T=1$
with $\tau=h^{3/2}$, $P^5$.}
\vspace{0.2cm}
\begin{small}
  \begin{tabular}{c|c|c|c|c|c|c|c}
\hline
&  &\multicolumn{2}{c|}{$\lambda=0.8$} & \multicolumn{2}{c|}{$\lambda=1.5$} & \multicolumn{2}{c}{$\lambda=2$}\\
\cline{3-4} \cline{5-6} \cline{7-8}
$\alpha$ & $h$ &$\|\cdot\|$-error  & order & $\|\cdot\|$-error & order & $\|\cdot\|$-error & order\\
\hline
\multirow{4}{*}{$0.3$}
   &1/10 & 3.949105e-06 &         &1.862967e-06 &        &1.089277e-06 &      \\
   &1/20 & 6.459789e-08 & 5.9339  &3.181162e-08 & 5.8719 &1.917998e-08 & 5.8276 \\
   &1/40 & 1.030441e-09 & 5.9702  &5.094751e-10 & 5.9644 &3.080510e-10 & 5.9603 \\
\hline
\multirow{4}{*}{$0.6$}
   &1/10 & 7.681143e-06 &         &3.623533e-06 &        &2.118680e-06 &     \\
   &1/20 & 1.271142e-07 & 5.9171  &6.259814e-08 & 5.8551 &3.774191e-08 & 5.8109 \\
   &1/40 & 2.017666e-09 & 5.9773  &9.975830e-10 & 5.9715 &6.031825e-10 & 5.9674 \\
\hline
\multirow{4}{*}{$0.9$}
   &1/10 & 1.101658e-05 &         &5.197007e-06 &        &3.038690e-06 &       \\
   &1/20 & 1.808903e-07 & 5.9284  &8.908050e-08 & 5.8664 &5.370876e-08 & 5.8221 \\
   &1/40 & 2.847323e-09 & 5.9894  &1.407785e-09 & 5.9836 &8.512089e-10 & 5.9795 \\
\hline
  \end{tabular}
\label{tabexmone6}
\end{small}
\end{table}
%%%%%%%%%%%%%%%%%%%%%%%%%%%%%
%%%%%%=======================================================
\begin{table}[htp]
\centering
\caption{The $L^2$ errors and convergence orders for different $\lambda$ with fixed $h=1/500$ and $P^2$.}
\vspace{0.2cm}
\begin{small}
  \begin{tabular}{c|c|c|c|c|c|c|c}
\hline
&  &\multicolumn{2}{c|}{$\lambda=0.8$} & \multicolumn{2}{c|}{$\lambda=1.5$} & \multicolumn{2}{c}{$\lambda=3$}\\
 \cline{3-4} \cline{5-6} \cline{7-8}
$q$ & $\tau$ &$\|\cdot\|$-error  & order & $\|\cdot\|$-error & order & $\|\cdot\|$-error & order\\
\hline
\multirow{4}{*}{$1$}
   &1/5  & 6.268259e-05 &         &3.254490e-05 &        &8.428462e-06 &        \\
   &1/10 & 2.943336e-05 & 1.0906  &1.563827e-05 & 1.0574 &4.264206e-06 & 0.9830  \\
   &1/20 & 1.418827e-05 & 1.0528  &7.640964e-06 & 1.0333 &2.143689e-06 & 0.9922  \\
   &1/40 & 6.955695e-06 & 1.0284  &3.773311e-06 & 1.0179 &1.074573e-06 & 0.9963  \\
 \hline
  \end{tabular}
\label{tabexmone7}
\end{small}
\end{table}
%%%%%=============================================
%%%%%%=======================================================
\begin{table}[htp]
\centering
\caption{The $L^2$ errors and convergence orders for different $\lambda$ with fixed $h=1/300$ and $P^3$.}
\vspace{0.2cm}
\begin{small}
  \begin{tabular}{c|c|c|c|c|c|c|c}
\hline
&  &\multicolumn{2}{c|}{$\lambda=0.8$} & \multicolumn{2}{c|}{$\lambda=1.5$} & \multicolumn{2}{c}{$\lambda=3$}\\
 \cline{3-4} \cline{5-6} \cline{7-8}
$q$ & $\tau$ &$\|\cdot\|$-error  & order & $\|\cdot\|$-error & order & $\|\cdot\|$-error & order\\
\hline
\multirow{4}{*}{$2$}
   &1/5  & 4.081267e-05 &         &2.158503e-05 &        &5.934015e-06 &        \\
   &1/10 & 1.131338e-05 & 1.8510  &6.164704e-06 & 1.8079 &1.817960e-06 & 1.7067  \\
   &1/20 & 2.989321e-06 & 1.9201  &1.657132e-06 & 1.8953 &5.083344e-07 & 1.8385  \\
   &1/40 & 7.690827e-07 & 1.9586  &4.302745e-07 & 1.9454 &1.347830e-07 & 1.9151  \\
 \hline
  \end{tabular}
\label{tabexmone8}
\end{small}
\end{table}
%%%%%=============================================
%%%%%%=======================================================
\begin{table}[htp]
\centering
\caption{The $L^2$ errors and convergence orders for different $\lambda$ with fixed $h=1/200$ and $P^4$.}
\vspace{0.2cm}
\begin{small}
  \begin{tabular}{c|c|c|c|c|c|c|c}
\hline
&  &\multicolumn{2}{c|}{$\lambda=0.8$} & \multicolumn{2}{c|}{$\lambda=1.5$} & \multicolumn{2}{c}{$\lambda=3$}\\
 \cline{3-4} \cline{5-6} \cline{7-8}
$q$ & $\tau$ &$\|\cdot\|$-error  & order & $\|\cdot\|$-error & order & $\|\cdot\|$-error & order\\
\hline
\multirow{4}{*}{$3$}
   &1/5  & 2.470358e-05 &         &1.349055e-05 &        &4.084818e-06 &        \\
   &1/10 & 3.798010e-06 & 2.7014  &2.168691e-06 & 2.6371 &7.345770e-07 & 2.4753  \\
   &1/20 & 5.249281e-07 & 2.8551  &3.072493e-07 & 2.8193 &1.108525e-07 & 2.7283  \\
   &1/40 & 6.894422e-08 & 2.9286  &4.087384e-08 & 2.9102 &1.523835e-08 & 2.8629  \\
 \hline
  \end{tabular}
\label{tabexmone9}
\end{small}
\end{table}
%%%%%=============================================
%%%%%%=======================================================
\begin{table}[htp]
\centering
\caption{The $L^2$ errors and convergence orders for different $\lambda$ with fixed $h=1/100$ and $P^5$.}
\vspace{0.2cm}
\begin{small}
  \begin{tabular}{c|c|c|c|c|c|c|c}
\hline
&  &\multicolumn{2}{c|}{$\lambda=0.8$} & \multicolumn{2}{c|}{$\lambda=1.5$} & \multicolumn{2}{c}{$\lambda=3$}\\
 \cline{3-4} \cline{5-6} \cline{7-8}
$q$ & $\tau$ &$\|\cdot\|$-error  & order & $\|\cdot\|$-error & order & $\|\cdot\|$-error & order\\
\hline
\multirow{4}{*}{$4$}
   &1/5  & 1.842158e-05 &         &1.074646e-05 &        &3.850820e-06 &        \\
   &1/10 & 1.452152e-06 & 3.6651  &9.187289e-07 & 3.5481 &4.024233e-07 & 3.2584  \\
   &1/20 & 1.010676e-07 & 3.8448  &6.674299e-08 & 3.7830 &3.267246e-08 & 3.6226  \\
   &1/40 & 6.652413e-09 & 3.9253  &4.489239e-09 & 3.8941 &2.325680e-09 & 3.8124  \\
 \hline
  \end{tabular}
\label{tabexmone10}
\end{small}
\end{table}
The $L^2$ errors and orders of the fully discrete LDG scheme (\ref{fullyLDG}) on uniform meshes are present in Table \ref{tabexmone1}-Table \ref{tabexmone10}. Table \ref{tabexmone1}-Table \ref{tabexmone5} list the $L^2$ errors and orders of accuracy for schemes (\ref{fullyLDG}) with different $k$ and fixed $\alpha=0.5$. In these tests we take $\tau=h^{(k+1)/q}$.
All the numerical results given in Table \ref{tabexmone1}-Table \ref{tabexmone5} are consistent with the theoretical analysis which presented in theorem \ref{theoremfullerr2}. Table \ref{tabexmone6} shows the errors and orders of scheme (\ref{fullyLDG}) for solving the problem \eqref{5.1} with the different parameters $\alpha$ and $\lambda$.
As expected, we observe that our scheme can achieve higher order accuracy in space, as well as in time. To test the high order of scheme (\ref{fullyLDG}) in time direction, we list the errors and orders of scheme (\ref{fullyLDG}) in Table \ref{tabexmone7}-Table \ref{tabexmone10}. We can again clearly observe the desired orders of accuracy from
these tables.
\\
%%%%%%%%%%%%%%%%%%%%%%%%%%%%%%%%%%%%%%%%%%%%%
\begin{example}\label{example2}
 In this example, we examine the following
 homogeneous  equation
\begin{equation}\label{5.2}
_{0}^{C}D_{t}^{\alpha,\lambda}u(x,t)= u_{xx}(x,t),(x,t)\in(0,1)\times(0,T],
\end{equation}
subjects to the boundary conditions
$u(0,t)=0,u(1,t)=0$,
and the initial value $u(x,0)=\sin(2\pi x).$
We can check that the exact solution of this initial-boundary value problem (\ref{5.2}) is
$u(x,t)=e^{-\lambda t}E_{\alpha}(- ~4\pi^2t^{\alpha})\sin(2\pi x),$
where the generalized Mittag-Leffler function $E_{\alpha}(\cdot)$ defined in  (\ref{defM-L}).
\end{example}
In this test, the finite element space is piecewise linear and
piecewise quadratic polynomials for the second and third order schemes, respectively.
The numerical results are shown in Table \ref{tabexmtwo1}-Table \ref{tabexmtwo2}.
The evolution of numerical solutions with different
  $\alpha, \lambda$ at different times are given in Fig. \ref{fig:example2a}.
%%%%%%%%%%%%%%%%%%%%%%%
%%%%%%======================================================
\begin{table}[htp]
\centering
\caption{The $L^2$ errors and orders of accuracy of problem \eqref{5.2} calculated by the fully LDG schemes (\ref{fullyLDG}) for different $\lambda$ with $\alpha=0.5$, $\tau=h^2$, and $T=1$.}
\vspace{0.2cm}
\begin{small}
  \begin{tabular}{c|c|c|c|c|c|c|c}
\hline
&  &\multicolumn{2}{c|}{$\lambda=0.5$} & \multicolumn{2}{c|}{$\lambda=2.5$} & \multicolumn{2}{c}{$\lambda=5$}\\
\cline{3-4} \cline{4-5}\cline{6-8}
$k$ & $h$ &$\|\cdot\|$-error  & order & $\|\cdot\|$-error & order & $\|\cdot\|$-error & order\\
\hline
\multirow{3}{*}{$ 1$}
   &1/5  & 5.969030e-04 &         &8.078203e-05 &        &6.630993e-06  &          \\
   &1/10 & 1.490511e-04 & 2.0017  &2.017187e-05 & 2.0017 &1.655808e-06  &2.0017     \\
   &1/20 & 3.725150e-05 & 2.0004  &5.041443e-06 & 2.0004 &4.138268e-07  &2.0004  \\
\hline
  \end{tabular}
\label{tabexmtwo1}
\end{small}
\end{table}
%%%%%%%%%%%%%%%%%%%%%%%%%%%%%%%%%%%%%
%%%%%%=======================================================
\begin{table}[htp]
\centering
\caption{The $L^2$ errors and orders of accuracy of problem \eqref{5.2} calculated by the fully LDG schemes (\ref{fullyLDG}) for different $\lambda$ with $\alpha=0.5$, $\tau=h^3$, and $T=1$.}
\vspace{0.2cm}
\begin{small}
  \begin{tabular}{c|c|c|c|c|c|c|c}
\hline
&  &\multicolumn{2}{c|}{$\alpha=0.1$} & \multicolumn{2}{c|}{$\alpha=0.5$} & \multicolumn{2}{c}{$\alpha=0.9$}\\
\cline{3-4} \cline{4-5}\cline{6-8}
$k$ & $h$ &$\|\cdot\|$-error  & order & $\|\cdot\|$-error & order & $\|\cdot\|$-error & order\\
\hline
\multirow{3}{*}{$ 2$}
   &1/5  & 2.100809e-05 &         &1.359967e-05 &        &6.184385e-06  &          \\
   &1/10 & 2.669664e-06 & 2.9762  &1.723894e-06 & 2.9798 &7.878049e-07  &2.9727     \\
   &1/20 & 3.351008e-07 & 2.9940  &2.163795e-07 & 2.9940 &1.080580e-07  &2.8660   \\
\hline
  \end{tabular}
\label{tabexmtwo2}
\end{small}
\end{table}
%%%%%%%=========================================
%%%%%%%=========================================
\begin{figure}[htp]
    \subfigure[$\alpha=0.5,t=1$]{
        \label{fig:mini:subfig:a} %% label for first subfigure
        \begin{minipage}[b]{0.5\textwidth}
            \centering
            \includegraphics[scale=.25]{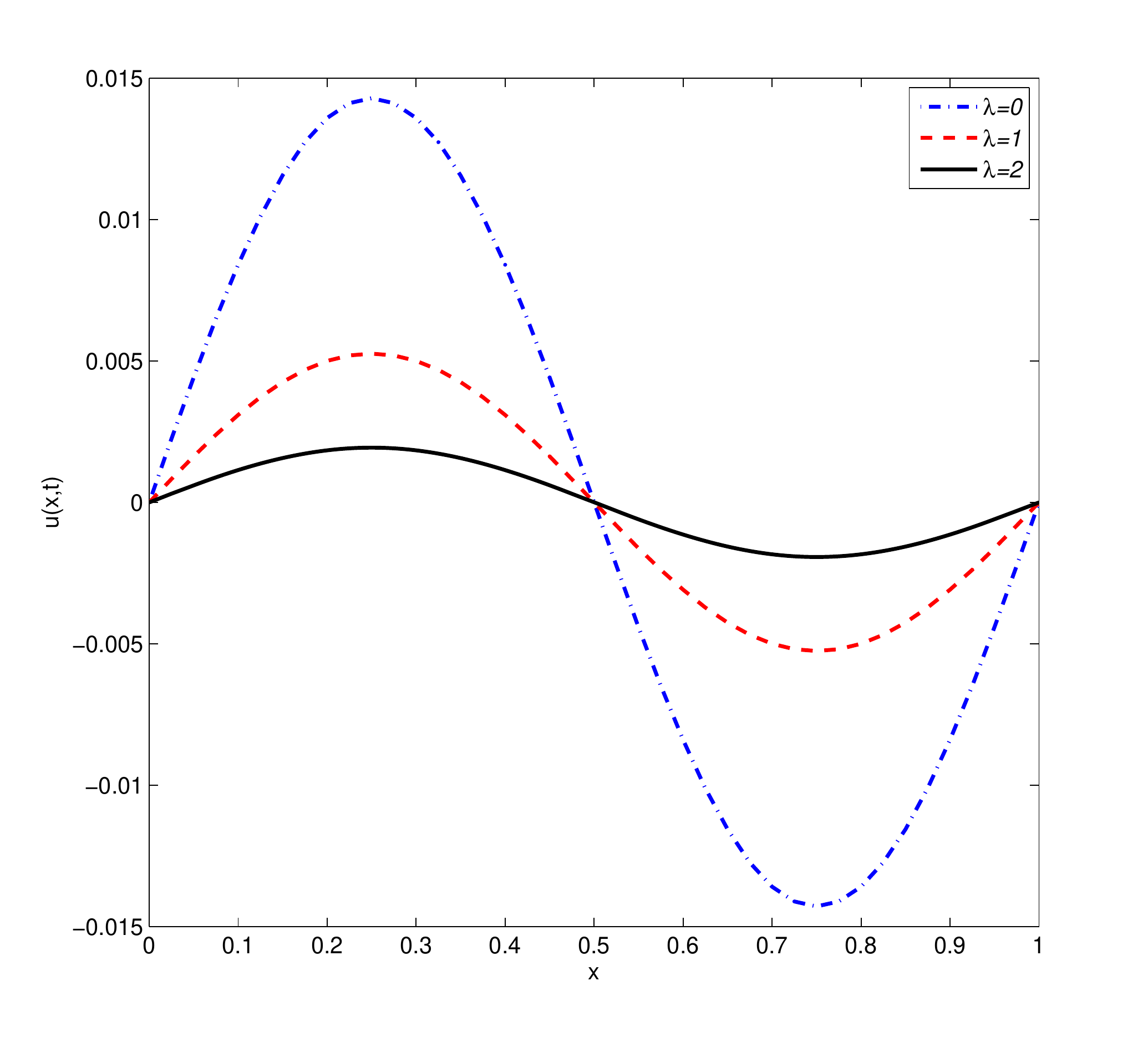}
        \end{minipage}
        } \hspace*{-20pt} %
    \subfigure[$\lambda=2,t=1$]{
        \label{fig:mini:subfig:b} %% label for second subfigure
        \begin{minipage}[b]{0.4\textwidth}
            \centering
            \includegraphics[scale=.25]{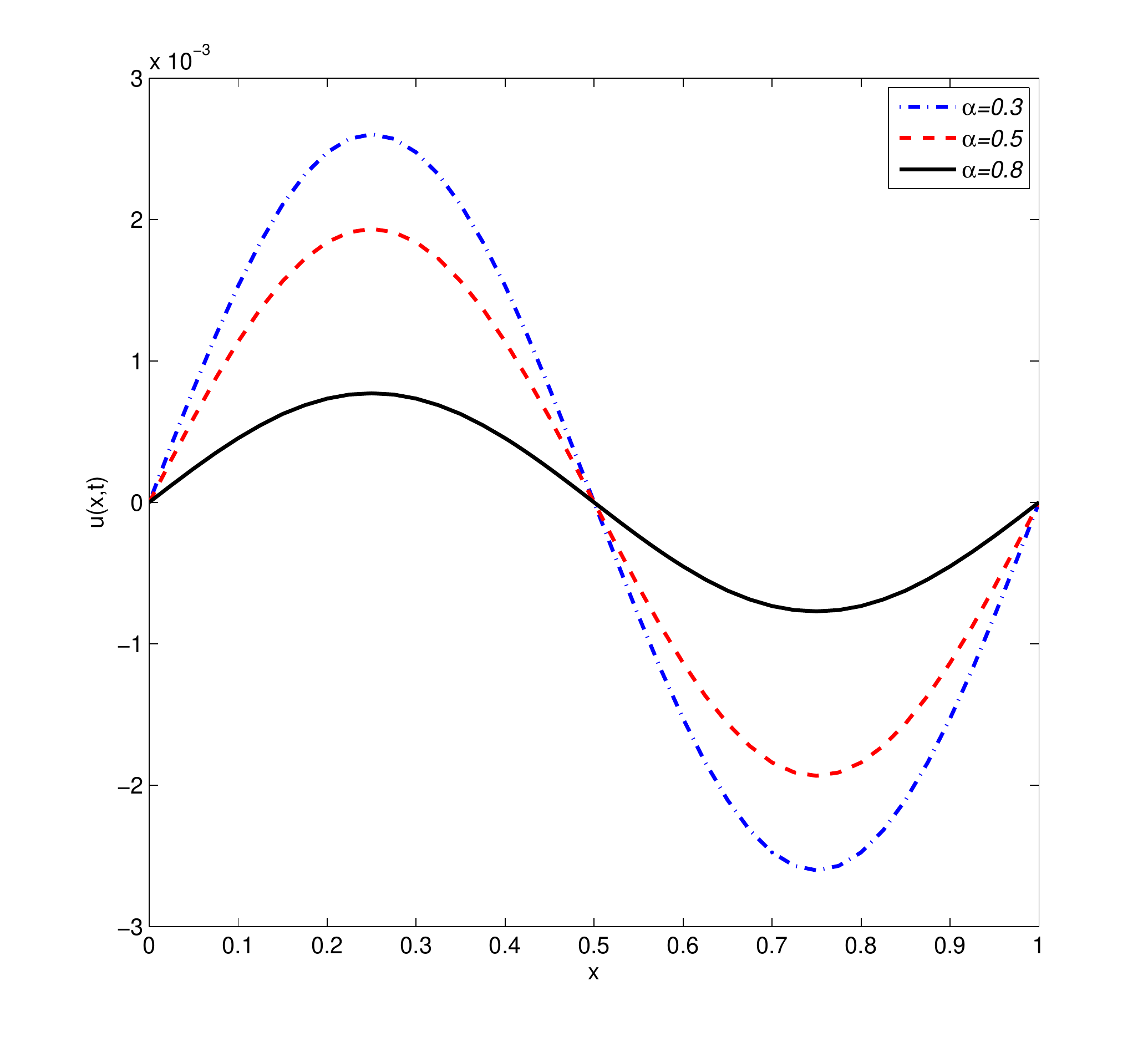}
        \end{minipage}
        }   \vspace*{-10pt}  \\
        \subfigure[$\alpha=0.5, \lambda=2$]{
        \label{fig:mini:subfig:b} %% label for second subfigure
        \begin{minipage}[b]{0.5\textwidth}
            \centering
            \includegraphics[scale=.25]{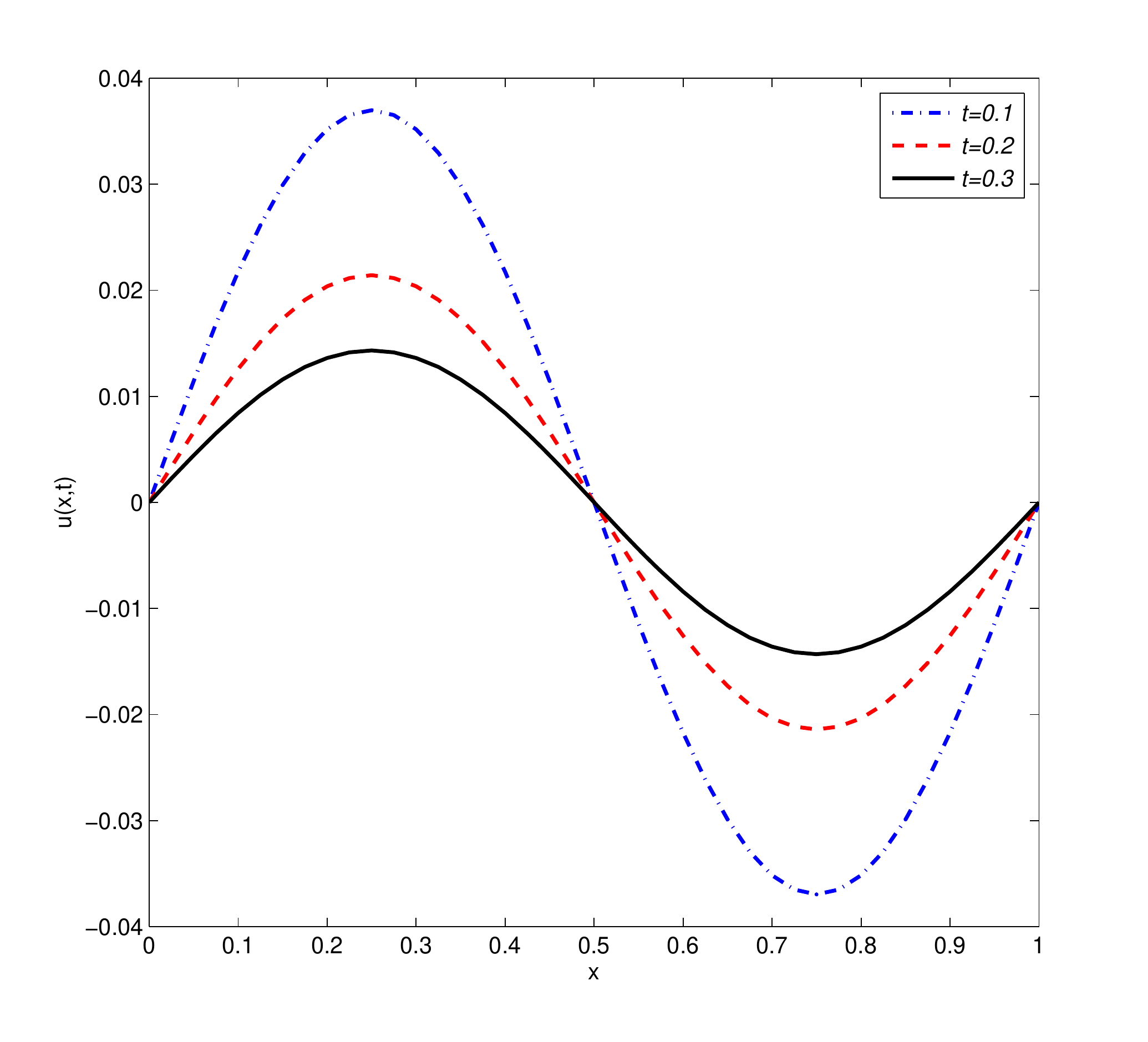}
        \end{minipage}
        }  \hspace*{-20pt}
        \subfigure[$\alpha=0.5, \lambda=2$]{
        \label{fig:mini:subfig:b} %% label for second subfigure
        \begin{minipage}[b]{0.4\textwidth}
            \centering
            \includegraphics[scale=.25]{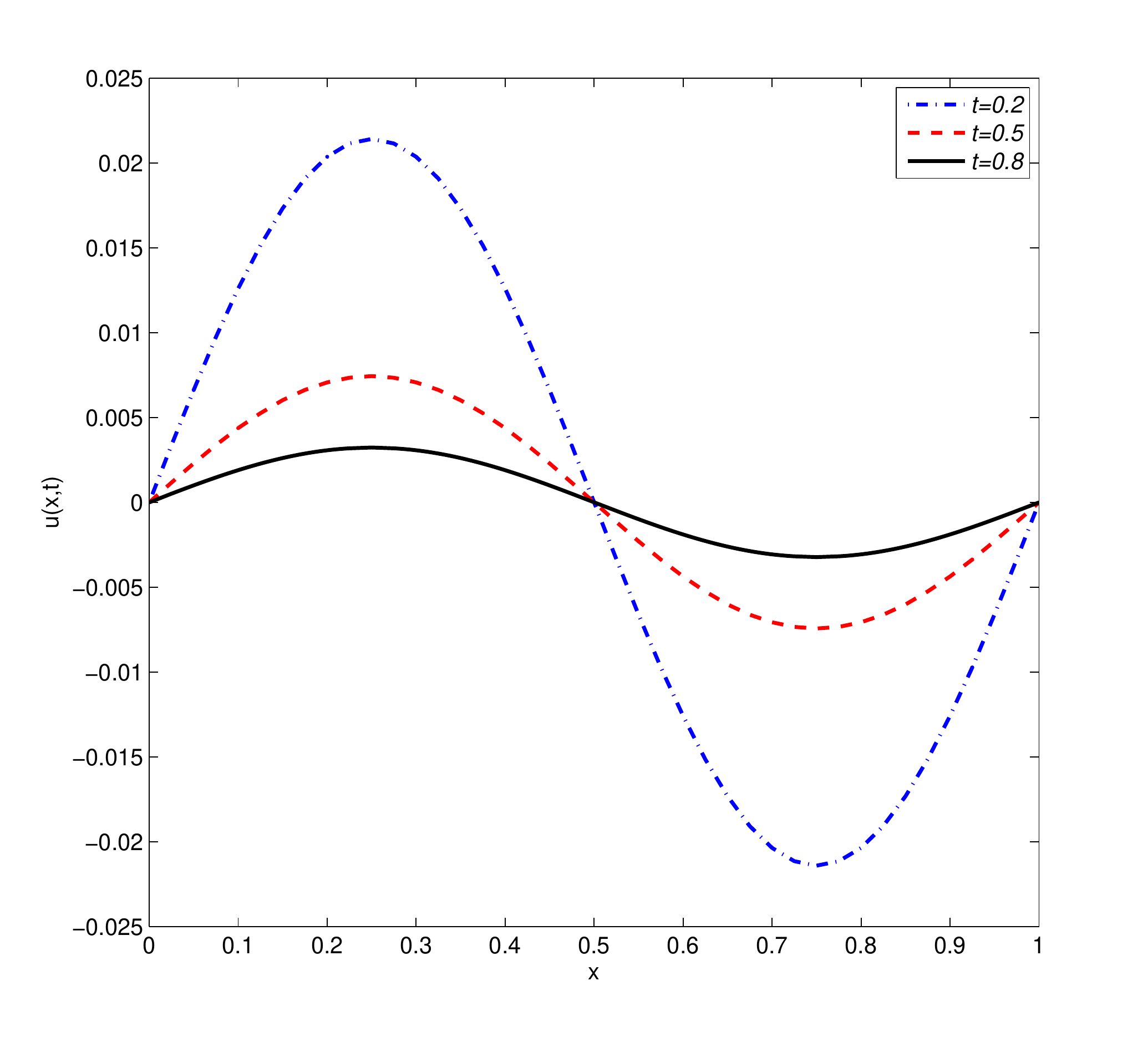}
        \end{minipage}
        }  \vspace*{-15pt}
    \caption{The evolution of $u(x,t)$ under different parameters $\alpha, \lambda$
    at different times.}\label{fig:example2a}
\end{figure}

%%%%%%%%%%%%%%%%%%%%%%%%%%%%%%%%%%%%%%%
\begin{example}\label{example3}
In this example, we will test the dynamics behavior of the tempered fractional diffusion equation (\ref{5.2}) with  homogeneous Dirichlet boundary conditions on a finite domain $[-4,4]$. We take the Gaussian function
\begin{equation}\label{5.3}
u(x,0)=\frac{1}{\sigma\sqrt{2\pi}}{\rm {exp}}\big( -\frac{x^2}{2\sigma^2}\big),
\end{equation}
as the initial condition.
\end{example}
The numerical results for this example are calculated by the fully discrete scheme (\ref{fullyLDG}). In the computation, we set $h=1/40, \tau=h^2, \sigma=0.01.$
The probability density function of a diffusion particle for different values of $\alpha, \lambda$ at different times are given in Fig. \ref{fig:example3}. It can be seen that, the different parameters $\alpha, \lambda, t$ have different effect for the probability density
of a particle, which is in agreement with the  analytic results  given in \cite{Henry:06,Langlands:08}.
 The effectiveness of our numerical schemes is confirmed once again.
%%%%%%%=========================================
\begin{figure}[thb]
    \subfigure[$\alpha=0.5, t=0.1$]{
        \label{fig:mini:subfig:a} %% label for first subfigure
        \begin{minipage}[b]{0.5\textwidth}
            \centering
            \includegraphics[scale=.25]{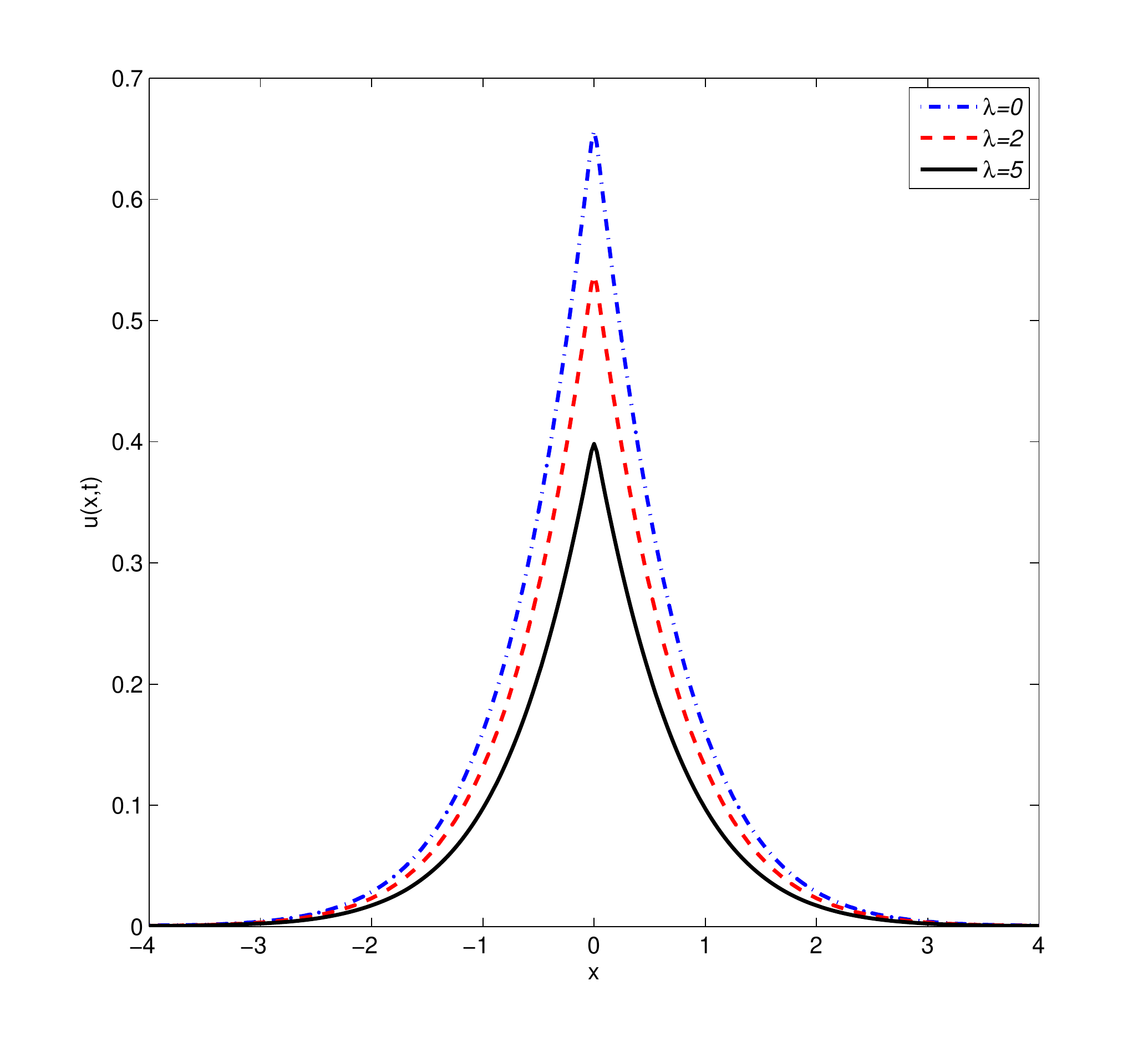}
        \end{minipage}
        } \hspace*{-20pt} %
    \subfigure[$\lambda=2, t=0.1$]{
        \label{fig:mini:subfig:b} %% label for second subfigure
        \begin{minipage}[b]{0.4\textwidth}
            \centering
            \includegraphics[scale=.25]{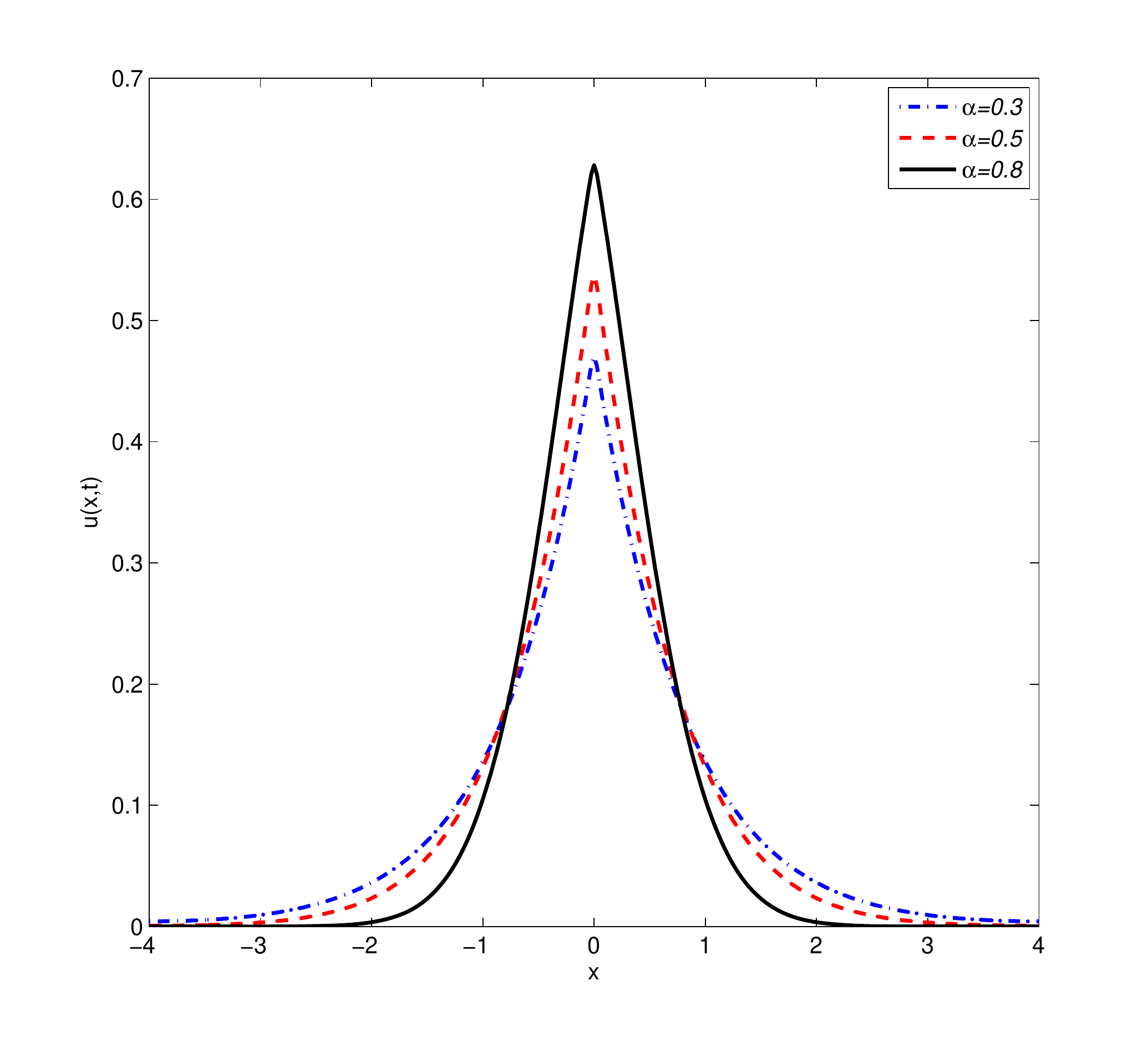}
        \end{minipage}
        }   \vspace*{-10pt}  \\
        \subfigure[$\alpha=0.5, \lambda=2$]{
        \label{fig:mini:subfig:b} %% label for second subfigure
        \begin{minipage}[b]{0.5\textwidth}
            \centering
            \includegraphics[scale=.25]{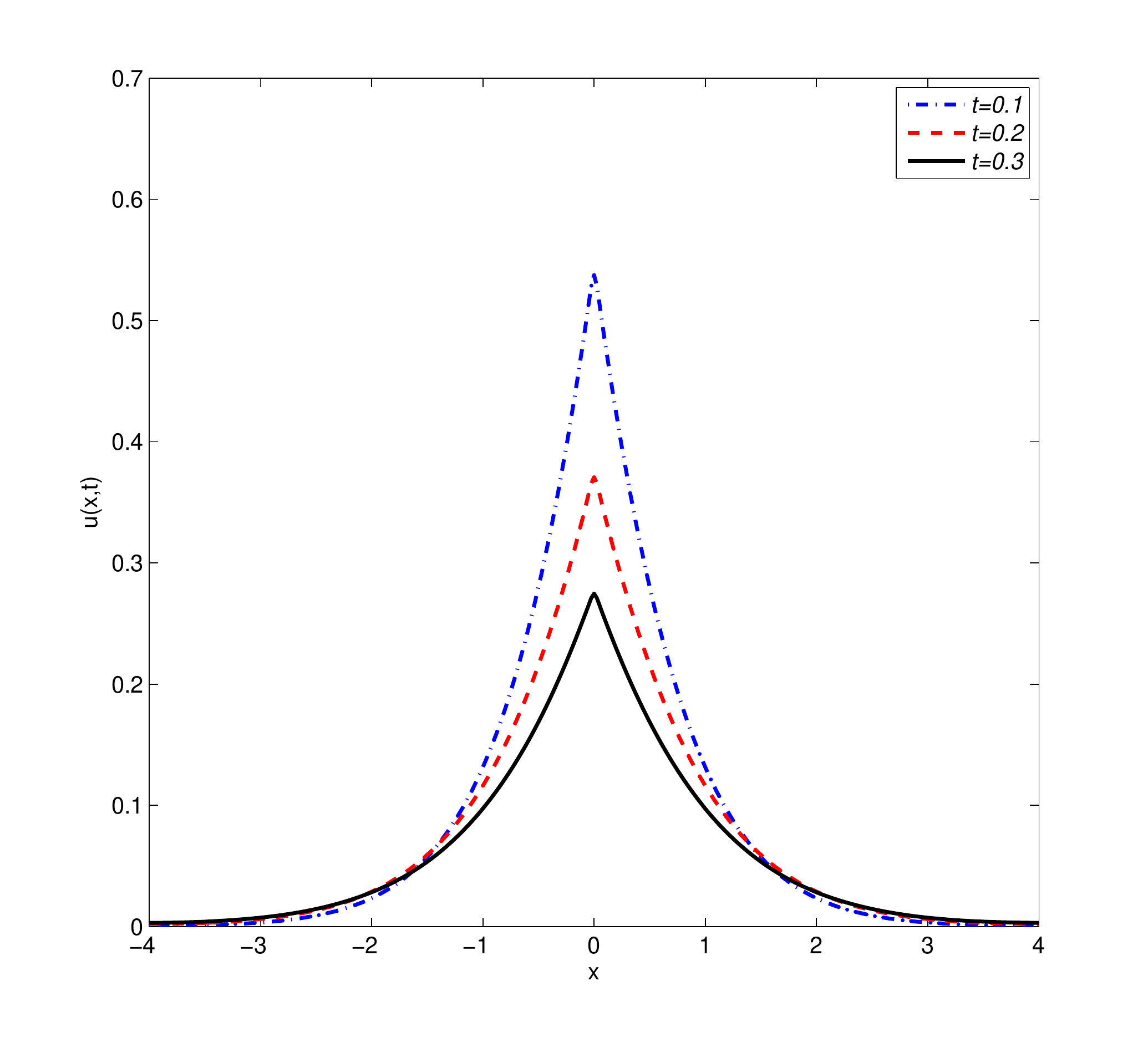}
        \end{minipage}
        }  \hspace*{-20pt}
        \subfigure[$\alpha=0.8, \lambda=2$]{
        \label{fig:mini:subfig:b} %% label for second subfigure
        \begin{minipage}[b]{0.4\textwidth}
            \centering
            \includegraphics[scale=.25]{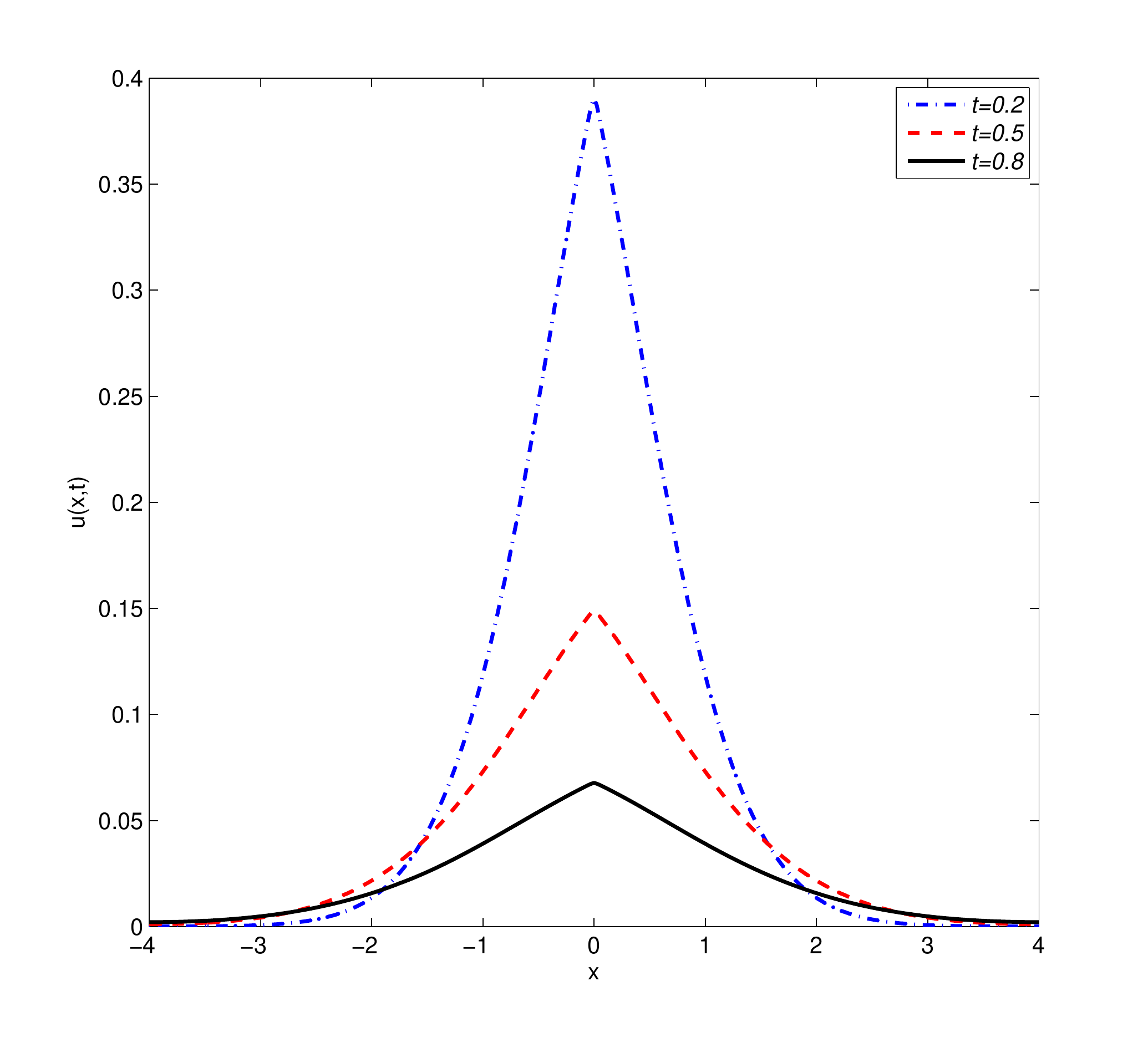}
        \end{minipage}
        }  \vspace*{-15pt}
    \caption{The evolution of $u(x,t)$ under different parameters $\alpha, \lambda$
    at different times.}\label{fig:example3}
\end{figure}
%%%%%%%%%%%%%%%%%%%%%%%%%%%%%%%%%%%%%%%
%%%%%%%%%%%%%%%%%%%%%%%%%%%%%%%%%%%%%%%
\section{Conclusions}\label{sec:6}
 We have presented a numerical method for a time fractional tempered diffusion equation. The proposed method is based on a combination of the weighted and shifted Lubich difference approaches in
the time direction and a LDG method in the space direction. The convergence
rate of the method is proven by providing a priori error estimate, and confirmed by
a series of numerical tests. It has been proved that
the proposed scheme is unconditionally stable and of $q$-order convergence in time and $k+1$-order convergence in
space. Some numerical experiments
have been carried out to support the theoretical results.
\\
\section*{Acknowledgments}  This research was partially
supported by the National Natural Science Foundation of China under Grant No.11426174, the Starting Research Fund from the Xi¡¯an university of Technology under Grant Nos. 108-211206, 2014CX022, the Natural
Science Basic Research Plan in Shaanxi Province of China under Grant No.2015JQ1022,
the Shaanxi science and technology research projects under Grant No.2015GY004.
%\bibliographystyle{abbrv}
%%
%\bibliography{xuzhangBib.bib}

\begin{thebibliography}{99}

\bibitem{Henry:06}B.I. Henry, T.A.M. Langlands, S.L. Wearne, Anomalous diffusion with linear reaction dynamics: From continuous time random walks to fractional reaction-diffusion equations, { Phys. Rev. E} {74}(2006) 031116.

\bibitem{Langlands:08}T.A. Langlands, B.I. Henry, S.L. Wearne,
 Anomalous subdiffusion with multispecies linear reaction dynamics,
 {Phys. Rev. E} {77} (2008) 021111.

\bibitem{Podlubny:99}I.~Podlubny,~Fractional differential equations,
 Academic Press, San Diego, 1999.

\bibitem{Sabzikar:15}
F. Sabzikar, M.M. Meerschaert,J.H. Chen, Tempered fractional calculus, J. Comput. Phys., 293 (2015) 14-28.

\bibitem{Li:14} C. Li, W.H. Deng,
 High order schemes for the tempered fractional diffusion equations, {Adv. Comput. Math.} {42} (2014) 543-572.

\bibitem{Meerschaert:04}
M.M. Meerschaert, C. Tadjeran, Finite difference approximations for fractional advection-dispersion flow equations, J. Comput. Appl. Math. 172 (1)
(2004) 65-77.

\bibitem{Sun:16} Z.Z. Sun, X.N. Wu, A fully discrete difference scheme for a diffusion-wave system, Appl. Numer. Math. 56 (2006) 193-209.

\bibitem{Murillo:15}J.Q. Murillo, S.B. Yuste,
 On three explicit difference schemes for fractional diffusion and diffusion-wave equations, {Phys. Scr.} {136} (2009) 14025-14030.

\bibitem{Liu:15} F.W. Liu, P.H. Zhuang, Q.X. Liu, The Applications and Numerical Methods of Fractional Differential Equations, Science Press, Beijing, 2015.

\bibitem{Sousa:15} E. Sousa, C. Li, A weighted finite difference method for the fractional diffusion equation based on the Riemann-Liouville derivative, Appl. Numer. Math. 90 (2015) 22-37.

\bibitem{Gracia:15} J.L. Gracia, M. Stynes, Central difference approximation of convection in Caputo fractional derivative two-point boundary value problems, J. Comput. Appl. Math. 273 (2015) 103-115.


\bibitem{Ervin:05} V.J. Ervin, J.P. Roop, Variational formulation for the stationary fractional advection dispersion equation, Numer. Methods for Partial Differential Equations 22(2005)  558-576.

\bibitem{Wang:14} H. Wang, D. Yang, S. Zhu, Inhomogeneous Dirichlet boundary-value problems of space-fractional diffusion equations and their finite element approximations. SIAM J. Numer. Anal.  52 (2014)  1292-1310.

\bibitem{Li:15a} C.P. Li, F.H. Zeng, Numerical methods for fractional calculus, CRC Press, Boca Raton, FL, 2015.

\bibitem{Zhao:15}
Y.M. Zhao, W.P. Bu, J.F. Huang, D.Y. Liu, Y.F. Tang, Finite element method for two-dimensional space-fractional advection-dispersion equations, Appl.
Math. Comput. 257 (2015) 553-565.

\bibitem{Jin:16} B. Jin, R. Lazarov, Z. Zhou, An analysis of the L1 scheme for the subdiffusion equation with nonsmooth data, IMA J. Numer. Anal. 36 (2016) 197-221.

\bibitem{Lin:07} Y.M. Lin, C.J. Xu, Finite difference/spectral approximations for the time-fractional diffusion equation, J. Comput. Phys. 225 (2007) 1533¨C1552.

\bibitem{Wang:16} S. Chen, J. Shen, L.L. Wang,
  Generalized Jacobi functions and their applications to fractional differential equations,  Math. Comp. 85 (2016) 1603-1638.

\bibitem{Baeumera:10}B. Baeumera, M.M. Meerschaert,  Tempered stable L\'{e}vy motion and transient super-diffusion, J. Comput. Appl. Math.  {233}(2010) 2438-2448.

\bibitem{Cartea:07a}
\'{A}. Cartea, D. del-Castillo-Negrete,  Fractional diffusion models of option prices in markets with jumps, Phys. A 374(2007) 749-763.

\bibitem{Zhang:17} H. Zhang, F. Liu, I. Turner, S. Chen, The numerical simulation of the tempered fractional Black-Scholes equation for European double barrier option, Appl. Math. Model. 40(2016) 5819-5834.

\bibitem{Hanert:02}E. Hanert, C. Piret,
 A Chebyshev pseudospectral method to solve the space-time tempered fractional diffusion equation, {SIAM J. Sci. Comput.} {36} (2014) 1797-1812.

\bibitem{Zayernouri:15}
M. Zayernouri, M. Ainsworth, G. Karniadakis, Tempered fractional Sturm-Liouville eigenproblems, SIAM J. Sci. Comput. 37 (4) (2015) A1777-A1800.

\bibitem{Huang:14}
C. Huang, Q. Song, Z.M. Zhang, Spectral collocation method for substantial fractional di
erential equations.\href{http://arxiv.org/abs/1408.5997v1}{arXiv:1408.5997v1 [math.NA] 26 Aug 2014}

\bibitem{Li:15} C. Li, W. H. Deng, L. Zhao,
 Well-posedness and numerical algorithm for the tempered fractional ordinary differential equations, \href{http://arxiv.org/abs/1501.00376v1}{arXiv:1501.00376v1 [math.CA] 2 Jan 2015}

\bibitem{Yu:06}Y.Y. Yu, W. H. Deng, Y.J. Wu,
 Third order difference schemes (without using points outside of the domain) for one sided space tempered fractional partial differential equations, {Appl. Numer. Math.} {112} (2017) 126-145.

\bibitem{Hao:17}
 Z. Hao, W. Cao, G. Lin, A second-order difference scheme for the  time fractional substantial diffusion equation, {J. Comput. Appl. Math.} {313} (2017) 54-69.


\bibitem{Zhao:16} L. Zhao, W. H. Deng, J. S. Hesthaven ,
 Spectral methods for tempered fractional differential equations,
  \href{http://arxiv.org/abs/1603.06511v1}{arXiv:1603.06511v1 [math.NA] 21 Mar 2016}



\bibitem{Tian:15} W. Y. Tian, H. Zhou, W. H. Deng,
 A class of second order difference approximations for solving space fractional diffusion equations, {Math. Comput.} {294} (2012) 1703-1727.


\bibitem{Chen:14}M. H. Chen, W. H. Deng,
 Fourth order difference approximations for space Riemann
 -Liouville derivatives based on weighted and shifted Lubich difference operators, {Commun. Comput. Phys.} {16} (2014) 516-540.

\bibitem{Chen:15} M. H. Chen, W. H. Deng, E. Barkai,
 Numerical algorithms for the forward and backward
 fractional Feynman-Kac equations, {J. Sci. Comput.} {62} (2015) 718-746.

\bibitem{Wang:10} H. Wang, K. Wang, T. Sircar, A direct $O(N \log2N)$ finite difference method for fractional diffusion equations, J. Comput. Phys. 229 (2010) 8095-8104.


\bibitem{Jiang:17} S. Jiang, J. Zhang, Q. Zhang, Z. Zhang, Fast evaluation of the Caputo fractional derivative and its applications to fractional diffusion equations. Commun. Comput. Phys. 21 (2017) 650-678.

\bibitem{Cockburn:98} B. Cockburn, C.-W. Shu,
 The local discontinuous Galerkin method for time-dependent
 convection-diffusion systems, {SIAM J. Numer. Anal.} {35} (1998) 2440-2463.

\bibitem{Hesthaven:08}
J.S. Hesthaven, T. Warburton,  Nodal Discontinuous Galerkin Methods. Algorithms, Analysis, and Applications. Springer, Berlin, 2008.

\bibitem{Cockburn:00} B. Cockburn, G. Karniadakis, C.-W. Shu,
 The development of discontinuous Galerkin methods, in Discontinuous Galerkin Methods: Theory, Computation and Applicatons, B. Cockburn G. Karniadakis and C.-W. Shu, editors,
 Lecture Notes in Computational Science and Engineering, volume 11, Springer, 2000, Part I: Overview, 3-50.

\bibitem{Xu:10}
Y. Xu, C.-W. Shu, Local discontinuous Galerkin methods for high-order time-dependent partial differential equations, Comm. Comput. Phys. 7 (2010) 1-46.

\bibitem{Shu:16}
C.-W. Shu,  High order WENO and DG methods for time-dependent convection-dominated PDEs: a brief survey of several recent developments , J. Comput. Phys. 316 (2016) 598-613.

\bibitem{Mustapha:10} K. Mustapha, W. McLean, Piecewise-linear, discontinuous Galerkin method for a fractional diffusion equation, Numer. Algorithms 56 (2010) 159-184.


\bibitem{Xu:13}Q. Xu, Z. Zheng, Discontinuous Galerkin method for time fractional
    diffusion equation, J. Informat. Comput. Sci. 10 (2013) 3253-3264.



\bibitem{Wei:12} L. Wei, X. Zhang, Y. He, S. Wang, Analysis of an implicit fully discrete local discontinuous Galerkin method for
the time-fractionalSchr \"{o}dingerequation, Finite Elem. Anal. Desi. 59(2012)28-34.



\bibitem{Guo:16} L. Guo, Z. B. Wang, S. Vong,
 Fully discrete local discontinuous Galerkin methods for some time-fractional fourth-order problems, {Int. J. Comput. Math.} {93} (2016) 1665-1682.

\bibitem{Liu:14} Y. Liu, M. Zhang, H.Li, J.C.Li,
 High-order local discontinuous Galerkin method combined with WSGD-approximation for a fractional subdiffusion equation, {Comput. Math. Appl.} {73} (2017) 1298-1314.



\bibitem{Ji:12} X. Ji, H. Z.Tang, High-order accurate Runge-Kutta (local) discontinuous Galerkin methods for one-and two-dimensional fractional diffusion equations, Numer. Math. Theor. Meth. Appl. 5(2012) 333-358.

\bibitem{Deng:13} W.H. Deng, J.S. Hesthaven, Local discontinuous Galerkin methods for fractional diffusion equations, ESAIM Math. Model. Numer. Anal. 47(2013)
1845-1864.

\bibitem{Alikhanov:01}
A.A. Alikhanov, {A priori estimates for solutions of boundary value problem for fractional-order equations},
Diff.Eq. {46} (2010) 660-666.

\bibitem{Dixo:86} J.Dixo, S. Mckee,  Weakly singular discrete gronwall inequalities, Z. angew. Math. Mech. 66(1986) 535-544.

\bibitem{Kirby:14} R. M. Kirby, G. E. Karniadakis,
 Selecting the numerical flux in discontinuous
Galerkin methods for diffusion problems, {J. Sci. Comput.} {22} (2005) 385-411.


\end{thebibliography}

\section*{Reference}
%%%% Bibliography  %%%%%%%%%%

\end{document}